\documentclass[psamsfonts, english]{amsart}

\usepackage{amssymb,amsfonts}
\usepackage[all,arc]{xy}
\usepackage{enumerate}
\usepackage{mathrsfs}
\usepackage{url}
\usepackage{hyperref}

\newtheorem{thm}{Theorem}[section]
\newtheorem*{thm*}{Theorem}
\newtheorem{cor}[thm]{Corollary}
\newtheorem{prop}[thm]{Proposition}
\newtheorem{lem}[thm]{Lemma}

\theoremstyle{definition}
\newtheorem{defn}[thm]{Definition}

\newtheorem{exmp}[thm]{Example}

\newtheorem{fact}[thm]{Fact}

\theoremstyle{remark}
\newtheorem{rem}[thm]{Remark}

\makeatletter
\let\c@equation\c@thm
\makeatother
\numberwithin{equation}{section}

\def\Ind#1#2{#1\setbox0=\hbox{$#1x$}\kern\wd0\hbox to 0pt{\hss$#1\mid$\hss}\lower.9\ht0\hbox to 0pt{\hss$#1\smile$\hss}\kern\wd0}
\def\Notind#1#2{#1\setbox0=\hbox{$#1x$}\kern\wd0\hbox to 0pt{\mathchardef
	\nn=12854\hss$#1\nn$\kern1.4\wd0\hss}\hbox to
	0pt{\hss$#1\mid$\hss}\lower.9\ht0 \hbox to
	0pt{\hss$#1\smile$\hss}\kern\wd0}
\newcommand{\ind}[1][]{\mathop{\mathpalette\Ind{}^{\!\!\!\!\rlap{$\scriptscriptstyle\textnormal{#1}$}\,\,\,\,}}}
\newcommand{\nind}[1][]{\mathop{\mathpalette\Notind{}^{\!\!\!\rlap{$\scriptscriptstyle\textnormal{#1}$}\,\,\,}}}

\usepackage{babel}
\providecommand{\claimname}{Claim}

\title{Transitivity of Kim-independence}

\date{\today}
\author{Itay Kaplan and Nicholas Ramsey}

\thanks{The first author would like to thank the Israel Science Foundation
for partial support of this research (grants No. 1533/14 and 1254/18).}

\begin{document}

\begin{abstract}
We prove several results on the behavior of Kim-independence upon changing the base in NSOP$_{1}$ theories.  As a consequence, we prove that Kim-independence satisfies transitivity and that this characterizes NSOP$_{1}$.  Moreover, we characterize witnesses to Kim-dividing as exactly the $\ind^{K}$-Morley sequences.  We give several applications, answering a number of open questions concerning transitivity, Morley sequences, and local character in NSOP$_{1}$ theories.  
\end{abstract}

\maketitle

\setcounter{tocdepth}{1}
\tableofcontents

\section{Introduction}

The class of NSOP$_{1}$ theories may be viewed as the class of theories that are simple at a generic scale.  This picture emerged piecemeal, starting with the results of Chernikov and the second-named author \cite{ArtemNick}, which established a Kim-Pillay-style criterion for NSOP$_{1}$ and characterized the NSOP$_{1}$ theories in terms of a weak variant of the independence theorem.  Simplicity-like behavior had been observed in certain algebraic structures\textemdash for example, the generic vector space with a bilinear form studied by Granger, and $\omega$-free PAC fields investigated by Chatzidakis\textemdash and these new results established these structures are NSOP$_{1}$ and suggested that this simplicity-like behavior might be characteristic of the class.  The analogy with simplicity theory was deepened in \cite{kaplan2017kim} and \cite{24-Kaplan2017} with the introduction of Kim-independence.  There it was shown that, in an NSOP$_{1}$ theory, Kim-independence satisfies appropriate versions of Kim's lemma, symmetry, the independence theorem, and local character and that, moreover, these properties individually characterize NSOP$_{1}$ theories.  This notion of independence has proved useful in proving preservation of NSOP$_{1}$ under various model-theoretic constructions and has been shown to coincide with natural algebraic notions of independence in new concrete examples.  In this way, the structure theory for NSOP$_{1}$ theories has developed along parallel lines to simplicity theory, with Kim-independence replacing the core notion of non-forking independence.  

The key difference between these settings stems from the fact that the notion of Kim-independence only speaks about the behavior of dividing at the \emph{generic} scale.  To say that $a$ is Kim-independent over $M$ with $b$ is to say that any $M$-indiscernible sequence $I$ beginning with $b$, if sufficiently generic over $M$, is conjugate over $Mb$ to one that is indiscernible over $Ma$.  In the initial definition of Kim-independence, genericity is understood to mean that the sequence is a Morley sequence in a global $M$-invariant type, but, after the fact, it turns out that broader notions of generic sequence give rise to equivalent definitions in the context of NSOP$_{1}$ theories \cite[Theorem 7.7]{kaplan2017kim}.  In any case, this additional genericity requirement in the definition of independence produces a curious phenomenon:  roughly speaking, asserting indiscernibility over a larger base is making a stronger statement, asserting genericity over a bigger base is making a weaker one.  This tension is what introduces subtleties in the generalization of facts from non-forking independence in simple theories to the broader setting of Kim-independence in NSOP$_{1}$ theories, as base monotonicity no longer holds.  In fact, an NSOP$_{1}$ theory in which Kim-independence satisfies base monotonicity is necessarily simple \cite[Proposition 8.8]{kaplan2017kim}.  

This paper is devoted to studying the ways that genericity over one base may be transfered to genericity over another base.  Base monotonicity trivializes all such questions in the context of non-forking independence in simple theories, so the issues we deal with here are new and unique to the NSOP$_{1}$ world.  The first work along these lines was in \cite{kruckman2018generic}, where Kruckman and the second-named author proved ``algebraically reasonable" versions of extension, the independence theorem, and the chain condition, which allow one to arrange for tuples to be Kim-independent over a given base and \emph{algebraically} independent over a larger one.  We build on this work, showing that in many cases one can arrange for Kim-independence over both bases and extend this to the construction of Morley sequences.  This leads to our main theorem:

\begin{thm*}
Suppose $T$ is a complete theory.  The following are equivalent:
\begin{enumerate}
\item $T$ is NSOP$_{1}$
\item Transitivity of Kim-independence over models:  if $M \prec N \models T$, $a \ind^{K}_{M} N$ and $a \ind^{K}_{N} b$, then $a \ind^{K}_{M} Nb$.\footnote{In the literature, \emph{transitivity} for a relation $\ind$ is sometimes taken to mean $a \ind_{A} b + a \ind_{Ab} c \iff a \ind_{A} bc$, which implies base monotonicity.  Since, in general, $\ind^{K}$ does not satisfy base monotonicity in an NSOP$_{1}$ theory, we use transitivity to denote only the $\implies$ direction.  This is reasonable since this may be paraphrased by saying that a non-Kim-forking extension of a non-Kim-forking extension is a non-Kim-forking extension (all extensions over models).  Kim has suggested using the term ``transitivity lifting" for this notion, but we opt for the simpler ``transitivity."}
\item $\ind^{K}$-Morley sequences over models are witnesses:  if $M \models T$ and $\varphi(x;b_{0})$ Kim-divides over $M$ and $\langle b_{i} : i < \omega \rangle$ is an $\ind^{K}$-Morley sequence, then $\{\varphi(x;b_{i}) : i < \omega\}$ is inconsistent.  
\end{enumerate}
\end{thm*}

The direction $(3) \implies (1)$ was known by \cite[Theorem 3.16]{kaplan2017kim}, but all other directions are new.  We prove $(1)\implies (2)$ in Theorem \ref{transitivity theorem}, $(2)\implies(1)$ in Proposition \ref{transitivity implies nsop1}, and finally $(1)\implies (3)$ in Theorem \ref{witnessing theorem} below.  

The theorem clarifies the extent to which concepts from simplicity theory can be carried over to the NSOP$_{1}$ context.  The Kim-Pillay theorem for simple theories catalogues the basic properties of non-forking independence in a simple theory.  We had showed all of these properties for Kim-independence except base monotonicity, transitivity, and local character in \cite{kaplan2017kim}, and observed that base monotonicity had to go for non-simple NSOP$_{1}$ theories.  Local character was later established in joint work with Shelah in \cite{24-Kaplan2017}, which left only transitivity.  An alternative formulation of transitivity, which is a consequence of the standard one and base monotonicity, was considered in \cite[Section 9.2]{kaplan2017kim}, where it was shown to fail in NSOP$_{1}$ theories in general.  The present theorem establishes transitivity in its usual form and, moreover, goes further, showing that transitivity of Kim-independence is characteristic of NSOP$_{1}$ theories.  

This theorem also represents a signficant technical development in the study of Kim-independence, allowing us to answer several questions.  The (1)$\implies$(2) direction and its proof settle two questions from our prior work \cite[Question 3.14, Question 3.16]{24-Kaplan2017}.  The (1)$\implies$(3) direction collapses two kinds of generic sequence studied in \cite{kaplan2017kim}:  it has as a corollary that tree Morley sequences coincide with total $\ind^{K}$-Morley sequences, answering \cite[Question 7.12]{kaplan2017kim} and, additionally, gives a characterization of witnesses for Kim-dividing in NSOP$_{1}$ theories.  

We give three applications in Section \ref{applications}.  First, we prove two `lifting lemmas' that show that, in an NSOP$_{1}$ theory, if $M$ is an elementary substructure of $N$, then whenever $a \ind^{K}_{M} N$, all $\ind^{K}$-Morley sequences and tree Morley sequences over $M$ beginning with $a$ are conjugate over $Ma$ to sequences that are respectively $\ind^{K}$-Morley or tree Morley over $N$.  This gives an analogue to a known result for non-forking Morley sequences in simple theories and clarifies the relationship between witnesses to Kim-dividing between two bases, one contained in another.  Secondly, we prove a local version of preservation of Kim-independence under unions of chains, which was previously only known for complete types.  In an NSOP$_{1}$ theory, a formula $k$-Kim-divides over an increasing union of models if and only if it $k$-Kim-divides over a cofinal collection of models in the chain (for an appropropriate definition of $k$-Kim-dividing), which answers \cite[Question 3.17]{24-Kaplan2017}.  Finally,  we reformulate the Kim-Pillay-style characterization of $\ind^{K}$ from \cite[Theorem 9.1]{24-Kaplan2017}, instead characterizing $\ind^{K}$ intrinsically in terms of properties of an abstract independence relation, without reference to finite satisfiability.  We expect that these results will have further applications in the study of this class of theories.

\section{Preliminaries}

Throughout the paper, $T$ will denote a complete theory in the language $L$ with infinite monster model $\mathbb{M} \models T$.  We will not notationally distinguish between elements and tuples.  We will write $x,y,$ and $z$ to denote tuples of variables, and use the letters $M,N$ to denote models of $T$.  

\subsection{NSOP$_{1}$ theories, invariant types, and Morley sequences}
\begin{defn}
\cite[Definition 2.2]{dvzamonja2004maximality} A formula $\varphi\left(x;y\right)$
has the $1$-\emph{strong order property} \emph{(SOP}$_{1})$ if there
is a tree of tuples $(a_{\eta})_{\eta\in2^{<\omega}}$ so
that 
\begin{itemize}
\item For all $\eta\in2^{\omega}$, the partial type $\{\varphi\left(x;a_{\eta\restriction n}\right): n<\omega\}$
is consistent. 
\item For all $\nu,\eta\in2^{<\omega}$, if $\nu\frown\langle0\rangle\unlhd\eta$
then $\left\{ \varphi\left(x;a_{\eta}\right),\varphi\left(x;a_{\nu\frown\langle1\rangle}\right)\right\} $
is inconsistent. 
\end{itemize}
A theory $T$ is \emph{NSOP}$_{1}$ if no formula has SOP$_{1}$ modulo
$T$. 
\end{defn}

The following equivalent formulation is more useful in practice:  

\begin{fact}
\label{karyversion} \cite[Lemma 5.1]{ArtemNick} \cite[Proposition 2.4]{kaplan2017kim} A theory $T$ has
NSOP$_{1}$ if and only if there is a formula $\varphi\left(x;y\right)$,
$k<\omega$, and an infinite sequence $\langle \overline{c}_{i}: i\in I\rangle$ with
$\overline{c}_{i}=\left(c_{i,0},c_{i,1}\right)$ satisfying: 
\begin{enumerate}
\item For all $i\in I$, $c_{i,0}\equiv_{\overline{c}_{<i}}c_{i,1}$. 
\item $\{\varphi\left(x;c_{i,0}\right) : i\in I\}$ is consistent. 
\item $\{\varphi\left(x;c_{i,1}\right): i\in I\}$ is $k$-inconsistent. 
\end{enumerate}
Moreover, if $T$ has SOP$_{1}$, there is such a $\varphi$ with $k=2$.  
\end{fact}

Given an ultrafilter $\mathcal{D}$ on a set of tuples $A$, we may define a complete type $\text{Av}(\mathcal{D},B)$ over $B$ by 
$$
\text{Av}(\mathcal{D}, B) = \{\varphi(x;b) : \{a \in A : \mathbb{M} \models \varphi(a,b)\} \in \mathcal{D}\}.
$$
We write $a\ind_{M}^{u}B$ to mean $\text{tp}\left(a/MB\right)$ is finitely satisfiable in $M$, in other words $\text{tp}(a/MB)$ is a \emph{coheir} of its restriction to $M$.  This is additionally equivalent to asserting that there is an ultrafilter $\mathcal{D}$ on tuples from $M$ such that $a \models \text{Av}(\mathcal{D},MB)$.  
%A type $p\in S\left(M\right)$ is an \emph{heir}
%of its restriction to $N\prec M$ if for every formula $\varphi\left(x;y\right)\in L\left(N\right)$
%and every $b\in M$, if $\varphi\left(x;b\right)\in p$ then $\varphi\left(x;b'\right)\in p$
%for some $b'\in N$. We denote this by $c\ind_{N}^{h}M$. This is equivalent to saying that $M\ind_{N}^{u}c$.  
A global type $q\in S\left(\mathbb{M}\right)$ is called $A$\emph{-invariant} if $b\equiv_{A}b'$ implies that, for all $\varphi(x;y)$, we have $\varphi\left(x;b\right)\in q$ if and
only if $\varphi\left(x;b'\right)\in q$. A global type $q$ is \emph{invariant}
if there is some small set $A$ such that $q$ is $A$-invariant.  
If $M$ is a model, then any type $p \in S(M)$ is finitely satisfiable in $M$ and hence $p =\text{Av}(\mathcal{D},M)$ for some ultrafilter $\mathcal{D}$ on tuples from $M$.  Then $\text{Av}(\mathcal{D},\mathbb{M})$ is a global $M$-finitely satisfiable (and hence $M$-invariant) extension of $p$ (see, e.g., \cite[Lemma VII.4.1]{shelah1990classification}).
%If $q\left(x\right)$ and $r\left(y\right)$ are $A$-invariant global
%types, then the type $\left(q\otimes r\right)\left(x,y\right)$ is
%defined to be $\text{tp}\left(a,b/\mathbb{M}\right)$ for any $b\models r$
%and $a\models q|_{\mathbb{M}b}$. It is also $A$-invariant. We define
%$q^{\otimes n}\left(x_{0},\ldots,x_{n-1}\right)$ by induction: $q^{\otimes1}=q$
%and $q^{\otimes n+1}=q\left(x_{n}\right)\otimes q^{\otimes n}\left(x_{0},\ldots,x_{n-1}\right)$.

\begin{defn}
Suppose $q$ is an $A$-invariant global type and $I$ is a linearly
ordered set. By a \emph{Morley sequence in }$q$ \emph{over} $A$
\emph{of order type} $I$, we mean a sequence $\langle b_{\alpha}: \alpha\in I\rangle$
such that for each $\alpha\in I$, $b_{\alpha}\models q|_{Ab_{<\alpha}}$
where $b_{<\alpha}=\langle b_{\beta} : \beta<\alpha \rangle$. Given a linear
order $I$, we will write $q^{\otimes I}$ for the unique global $A$-invariant
type in variables $\langle x_{\alpha}: \alpha \in I \rangle$ such that for any
$B\supseteq A$, if $\overline{b}\models q^{\otimes I}|_{B}$ then
$b_{\alpha}\models q|_{Bb_{<\alpha}}$ for all $\alpha\in I$. If
$q$ is, moreover, finitely satisfiable in $A$, in which case $b_{\alpha} \ind^{u}_{A} b_{<\alpha}$ for all $\alpha \in I$, then we refer to
a Morley sequence in $q$ over $A$ as a \emph{coheir sequence} over
$A$.
\end{defn}

We will also make use of the dual notions of heir and an heir sequence:

\begin{defn}
If $B \supseteq M$, we say that $p \in S(B)$ is an \emph{heir} of its restriction to $M$ if $B \ind^{u}_{M} a$ for some, equivalently all, $a \models p$ and we write $a \ind^{h}_{M} b$ if and only if $\text{tp}(a/Mb)$ is an heir of $\text{tp}(a/M)$ if and only if $b \ind^{u}_{M} a$.  We say that $\langle b_{i} : i \in I \rangle$ is an \emph{indiscernible heir sequence} over $M$ if $\langle b_{i} : i \in I \rangle$ is $M$-indiscernible and $b_{i} \ind^{h}_{M} b_{<i}$ for all $i \in I$.  
\end{defn}

\begin{defn}
Suppose $M$ is a model.  
\begin{enumerate}
\item We say that $\varphi\left(x;b\right)$ \emph{Kim-divides over }$M$ if there is a global $M$-invariant $q\supseteq\text{tp}\left(b/M\right)$ and Morley sequence $\langle b_{i} : i < \omega \rangle$ over $M$ in $q$ with $\{\varphi(x;b_{i}): i < \omega\}$ inconsistent.  
\item We say that $\varphi\left(x;b\right)$ Kim-forks over $M$ if it implies
a finite disjunction of formulas, each Kim-dividing over $M$. 
\item A type $p$ Kim-forks over $M$ if there is $\varphi(x;b)$ such that $p \vdash \varphi(x;b)$ and $\varphi(x;b)$ Kim-forks over $M$.  
\item We write $a\ind_{M}^{K}B$ for $\text{tp}\left(a/MB\right)$ does not Kim-fork over $M$.  We may also paraphrase $a \ind^{K}_{M} B$ as $a$ and $B$ are \emph{Kim-independent} over $M$.  
\item We say that an infinite sequence $\langle a_{i} : i \in I \rangle$ is an $\ind^{K}$\emph{-Morley sequence} over $M$ if $\langle a_i : i \in I \rangle$ is $M$-indiscernible and $a_{i} \ind^{K}_{M} a_{<i}$ for all $i \in I$.  
\end{enumerate}
\end{defn}

Note that if $a\ind_{M}^{u}B$ then $a\ind_{M}^{f}B$ (i.e. $\text{tp}\left(a/BM\right)$
does not fork over $M$) which implies $a\ind_{M}^{K}B$. 

Kim-independence may be used to give several equivalents of NSOP$_{1}$.  In order to state the appropriate form of local character for this notion, we will need to introduce the generalized club filter.  

\begin{defn}
\label{clubdef} Let $\kappa$ be a cardinal and $X$ a set with $\left|X\right|\geq\kappa$.
We write $\left[X\right]^{\kappa}$ to denote $\{Y\subseteq X : \left|Y\right|=\kappa\}$ and likewise $[X]^{<\kappa}$ for $\bigcup_{\lambda < \kappa} [X]^{\lambda}$.  A set $C\subseteq\left[X\right]^{\kappa}$ is \emph{club} if, for every $Y\in\left[X\right]^{\kappa}$, there is some $Z\in C$
with $Y\subseteq Z$ and if, whenever $\langle Y_{i} : i<\alpha\leq\kappa \rangle$ is an increasing chain in $C$, i.e. each
$Y_{i}\in C$ and $i<j<\alpha$ implies $Y_{i}\subseteq Y_{j}$, then
$\bigcup_{i<\alpha}Y_{i}\in C$. 
\end{defn}

\begin{fact} \label{basic kimindep facts}
\cite[Theorem 8.1]{kaplan2017kim} \cite[Theorem 1.1]{24-Kaplan2017} \label{kimslemma} The following
are equivalent for the complete theory $T$: 
\begin{enumerate}
\item $T$ is NSOP$_{1}$. 
\item Kim's lemma for Kim-dividing: Given any model $M\models T$ and formula
$\varphi\left(x;b\right)$, $\varphi\left(x;b\right)$ Kim-divides if and only if for any $\langle b_{i} : i < \omega \rangle$ Morley over $M$ in some global $M$-invariant type, $\{\varphi\left(x;b_{i}\right) : i < \omega\}$ is inconsistent.
\item Symmetry of Kim independence over models: $a\ind_{M}^{K}b$ iff $b\ind_{M}^{K}a$
for any $M\models T$.
\item Local character on a club:  given any model $M \models T$ and type $p \in S(M)$, the set $\{N \prec M : |N| = |T| \text{ and } p \text{ does not Kim-divide over }N\}$ is a club subset of $[M]^{|T|}$.  
\item Independence theorem over models: if $A\ind^{K}_{M}B$, $c\ind^{K}_{M}A$,
$c'\ind^{K}_{M}B$ and $c\equiv_{M}c'$ then there is some $c''\ind^{K}_{M}AB$
such that $c''\equiv_{MA}c$ and $c''\equiv_{MB} c'$.
\end{enumerate}
\end{fact}

\begin{rem} \label{local character for bigger cardinals}
Because NSOP$_{1}$ is preserved by naming constants, we also see that if $\kappa \geq |T|$ and we are given any model $M \models T$ with $|M| \geq \kappa$ and type $p \in S(M)$, the set $\{N \prec M : |N| = \kappa \text{ and } p \text{ does not Kim-divide over }N\}$ is a club subset of $[M]^{\kappa}$.  This follows by choosing an arbitrary $M_{0} \prec M$ of size $\kappa$ and applying Fact \ref{basic kimindep facts}(3) to the theory $T(M_{0})$ obtained from $T$ by adding constants for $M_{0}$.  
\end{rem}

We will make extensive use of the following additional properties of Kim-independence in NSOP$_{1}$ theories:

\begin{fact} \label{fact:Kim Morley is consistent}
Suppose that $T$ is NSOP$_{1}$ and $M \models T$. 
\begin{enumerate}
\item Extension:  if $a \ind^{K}_{M} b$, then given any $c$, there is $a' \equiv_{Mb} a$ such that $a' \ind^{K}_{M} bc$ \cite[Proposition 3.20]{kaplan2017kim}.
\item Consistency along $\ind^{K}$-Morley sequences:  suppose $\langle a_{i} : i<\omega \rangle$ is
an $\ind^{K}$-Morley sequence over $M$.  Then if $\varphi\left(x,a_{0}\right)$
does not Kim-divide over $M$, then $\{\varphi\left(x,a_{i}\right) : i<\omega\}$
does not Kim-divide over $M$, and in particular it is consistent \cite[Lemma 7.6]{kaplan2017kim}.
\item Strengthened independence theorem:  Suppose $c_{0} \equiv_{M} c_{1}$, $c_{0} \ind^{K}_{M} a$, $c_{1} \ind^{K}_{M} b$ and $a \ind^{K}_{M} b$.  Then there is $c \models \text{tp}(c_{0}/Ma) \cup \text{tp}(c_{1}/Mb)$ such that $a \ind^{K}_{M} bc$, $b \ind^{K}_{M} ac$, and $c \ind^{K}_{M} ab$ \cite[Theorem 2.13]{kruckman2018generic}.
\end{enumerate}
\end{fact}

We will need the following chain condition for $\ind^{K}$-Morley sequences, which is a slight elaboration of the proof of Fact \ref{fact:Kim Morley is consistent}(2).  

\begin{lem} \label{chain condition for indk}
Suppose $T$ is NSOP$_{1}$ and $M \models T$.  If $a \ind^{K}_{M} b_{0}$ and $I = \langle b_{i} : i < \omega \rangle$ is an $\ind^{K}$-Morley sequence over $M$, then there is $a' \equiv_{Mb_{0}} a$ such that $I$ is $Ma'$-indiscernible and $a' \ind^{K}_{M} I$.
\end{lem}

\begin{proof}
Let $p(x;b_{0}) = \text{tp}(a/Mb_{0})$.  By induction, we will choose $a_{n}$ such that $a_{n} \models \bigcup_{i \leq n} p(x;b_{i})$ and $a_{n} \ind^{K}_{M} b_{\leq n}$.  For $n = 0$, we put $a_{0} = a$.  Given $a_{n}$, pick $a'$ such that $a'b_{n+1} \equiv_{M} ab_{0}$.  Then, by invariance, we have $a' \ind^{K}_{M} b_{n+1}$ and, additionally, $b_{n+1} \ind^{K}_{M} b_{\leq n}$, and $a_{n} \ind^{K}_{M} b_{\leq n}$.  As $a' \equiv_{M} a \equiv_{M} a_{n}$, we may apply the independence theorem to find $a_{n+1}$ such that $a_{n+1} \equiv_{Mb_{\leq n}} a_{n}$, $a_{n+1} \equiv_{Mb_{n+1}} a'$, and $a_{n+1} \ind^{K}_{M} b_{\leq n+1}$.  In particular, $a_{n+1} \models \bigcup_{i \leq n+1} p(x;b_{i})$, completing the induction.  

By compactness and finite character, we can find $a_{*} \models \bigcup_{i < \omega} p(x;b_{i})$ such that $a_{*} \ind^{K}_{M} I$.  By compactness, Ramsey, an automorphism, we may assume $I$ is $Ma_{*}$-indiscernible, completing the proof.
\end{proof}

\subsection{Generalized indiscernibles and a class of trees}

The construction of tree Morley sequences goes by way of an inductive construction of approximations to Morley trees indexed by a certain class of trees.  Although the initial set-up is somewhat cumbersome, the definitions allow us to give simple and streamlined constructions.  It will be convenient to use the notation and basic definitions that accompany the trees $\mathcal{T}_{\alpha}$ from \cite[Section 5.1]{kaplan2017kim}.  The subsection below consists entirely of this notation and these definitions which are reproduced for the readers' convenience.  

For an ordinal $\alpha$, let the language $L_{s,\alpha}$ be $\langle \unlhd, \wedge, <_{lex}, (P_{\beta})_{\beta \leq \alpha} \rangle$.  We may view a tree with $\alpha$ levels as an $L_{s,\alpha}$-structure by interpreting $\unlhd$ as the tree partial order, $\wedge$ as the binary meet function, $<_{lex}$ as the lexicographic order, and $P_{\beta}$ interpreted to define level $\beta$.  Our trees will be understood to be an $L_{s,\alpha}$-structure for some appropriate $\alpha$.  We recall the definition of a class of trees $\mathcal{T}_{\alpha}$ below:

\begin{defn}
Suppose $\alpha$ is an ordinal.  We define $\mathcal{T}_{\alpha}$ to be the set of functions $f$ such that 
\begin{itemize}
\item $\text{dom}(f)$ is an end-segment of $\alpha$ of the form $[\beta,\alpha)$ for $\beta$ equal to $0$ or a successor ordinal.  If $\alpha$ is a successor, we allow $\beta = \alpha$, i.e. $\text{dom}(f) = \emptyset$.
\item $\text{ran}(f) \subseteq \omega$.
\item finite support:  the set $\{\gamma \in \text{dom}(f) : f(\gamma) \neq 0\}$ is finite.    
\end{itemize}
We interpret $\mathcal{T}_{\alpha}$ as an $L_{s,\alpha}$-structure by defining 
\begin{itemize}
\item $f \unlhd g$ if and only if $f \subseteq g$.  Write $f \perp g$ if $\neg(f \unlhd g)$ and $\neg(g \unlhd f)$.  
\item $f \wedge g = f|_{[\beta, \alpha)} = g|_{[\beta, \alpha)}$ where $\beta = \text{min}\{ \gamma : f|_{[\gamma, \alpha)} =g|_{[\gamma, \alpha)}\}$, if non-empty (note that $\beta$ will not be a limit, by finite support). Define $f \wedge g$ to be the empty function if this set is empty (note that this cannot occur if $\alpha$ is a limit).  
\item $f <_{lex} g$ if and only if $f \vartriangleleft g$ or, $f \perp g$ with $\text{dom}(f \wedge g) = [\gamma +1,\alpha)$ and $f(\gamma) < g(\gamma)$
\item For all $\beta \leq \alpha$, $P_{\beta} = \{ f \in \mathcal{T}_{\alpha} : \text{dom}(f) = [\beta, \alpha)\}$.  
\end{itemize}
\end{defn}

\begin{rem}
Condition (1) in the definition of $\mathcal{T}_{\alpha}$ was stated incorrectly in the first arXiv version of \cite{kaplan2017kim} via the weaker requirement that $\text{dom}(f)$ is an end-segment, non-empty if $\alpha$ is limit.  There, and below, the inductive constructions assume that $\mathcal{T}_{\alpha+1}$ consists of the empty function (the root) and countably many copies of $\mathcal{T}_{\alpha}$ given by $\{\langle i \rangle \frown \eta : i< \omega,\eta \in \mathcal{T}_{\alpha}\}$ (where this concatenation is defined below in Definition \ref{concatenation}).  But if $\alpha$ is a limit, this becomes false if we allow functions with domain $\{\alpha\}$ since the empty function is not an element of $\mathcal{T}_{\alpha}$ and therefore the function $\alpha \mapsto i$ is not of the form $\langle i \rangle \frown \eta$ for some $\eta \in \mathcal{T}_{\alpha}$.  This is rectified by omitting functions whose domain is an end-segment of the form $[\beta,\alpha)$ for $\beta$ limit.	\end{rem}

\begin{defn} \label{concatenation}
Suppose $\alpha$ is an ordinal.  
\begin{enumerate}
\item (Restriction) If $w \subseteq \alpha \setminus \text{lim}(\alpha)$, the \emph{restriction of} $\mathcal{T}_{\alpha}$ \emph{to the set of levels  }$w$ is given by 
$$
\mathcal{T}_{\alpha} \upharpoonright w = \{\eta \in \mathcal{T}_{\alpha} : \min (\text{dom}(\eta)) \in w \text{ and }\beta \in \text{dom}(\eta) \setminus w \implies \eta(\beta) = 0\}.
$$
\item (Concatenation)  If $\eta \in \mathcal{T}_{\alpha}$, $\text{dom}(\eta) = [\beta+1,\alpha)$ for some $\beta \in \alpha \setminus \text{lim}(\alpha)$, and $i < \omega$, let $\eta \frown \langle i \rangle$ denote the function $\eta \cup \{(\beta,i)\}$.  We define $\langle i \rangle \frown \eta \in \mathcal{T}_{\alpha+1}$ to be $\eta \cup \{(\alpha,i)\}$.  When we write $\langle i \rangle \in \mathcal{T}_{\alpha+1}$ by itself, we use this to denote the function $\{(\alpha,i)\}$.
\item (Canonical inclusions) If $\alpha < \beta$, we define the map $\iota_{\alpha \beta} : \mathcal{T}_{\alpha} \to \mathcal{T}_{\beta}$ by $\iota_{\alpha \beta}(f) = f \cup \{(\gamma, 0) : \gamma \in \beta \setminus \alpha\}$.
\item (The all $0$'s path) If $\beta < \alpha$, then $\zeta_{\beta}$ denotes the function with $\text{dom}(\zeta_{\beta}) = [\beta, \alpha)$ and $\zeta_{\beta}(\gamma) = 0$ for all $\gamma \in [\beta,\alpha)$.  This defines an element of $\mathcal{T}_{\alpha}$ if and only if $\beta \in \alpha \setminus \text{lim}(\alpha)$.  
\end{enumerate}
\end{defn}

We will most often be interested in collections of tuples indexed by $\mathcal{T}_{\alpha}$ and, if $(a_{\eta})_{\eta \in \mathcal{T}_{\alpha}}$ is such a collection and $\eta \in \mathcal{T}_{\alpha}$, we will write $a_{\unrhd \eta}$ and $a_{\vartriangleright \eta}$ for tuples enumerating the elements indexed by elements of $\mathcal{T}_{\alpha}$ above or strictly above $\eta$ in the tree partial order, respectively.  Note that if $\beta < \alpha$ is a limit ordinal and $\eta \in \mathcal{T}_{\alpha}$ has $\text{dom}(\eta) = [\beta+1,\alpha)$, then $\beta \frown \langle i \rangle$ is a function whose domain is $[\beta,\alpha)$ and is therefore not in $\mathcal{T}_{\alpha}$.  If $(a_{\eta})_{\eta \in \mathcal{T}_{\alpha}}$ is a collection of tuples indexed by $\mathcal{T}_{\alpha}$, we will abuse notation and write $a_{\unrhd \eta \frown \langle i \rangle}$ for the tuple that enumerates $\{a_{\nu} : \nu \in \mathcal{T}_{\alpha}, \eta \frown \langle i \rangle \subseteq \nu\}$ and likewise for $a_{\unrhd \zeta_{\beta}}$.  

We additionally remark that the concatention notation is only unambiguous once we have specified in which tree the element lives\textemdash for example, $\langle i \rangle \frown \langle j \rangle$ can denote an element of $\mathcal{T}_{\alpha+2}$ when $\langle j \rangle \in \mathcal{T}_{\alpha+1}$ or an element of $\mathcal{T}_{\alpha+1}$ if $\langle i \rangle \in \mathcal{T}_{\alpha+1}$, but this notation reads unambiguously once we have specified in which tree we are referring to $\langle i \rangle \frown \langle j \rangle$.  In the arguments below, the intended meaning of concatenation is clear from context and no confusion will arise.  

The function $\iota_{\alpha \beta}$ includes $\mathcal{T}_{\alpha}$ into $\mathcal{T}_{\beta}$ by adding zeros to the bottom of every node in $\mathcal{T}_{\alpha}$.  Clearly if $\alpha < \beta < \gamma$, then $\iota_{\alpha \gamma} = \iota_{\beta \gamma} \circ \iota_{\alpha \beta}$.  If $\beta$ is a limit, then $\mathcal{T}_{\beta}$ is the direct limit of the $\mathcal{T}_{\alpha}$ for $\alpha < \beta$ along these maps.  

\begin{defn}  Suppose $I$ is an $L'$-structure, where $L'$ is some language. 
\begin{enumerate}
\item  We say that $(a_{i} : i \in I)$ is a set of $I$\emph{-indexed indiscernibles over} $A$ if whenever 

$(s_{0}, \ldots, s_{n-1})$, $(t_{0}, \ldots, t_{n-1})$ are tuples from $I$ with 
$$
\text{qftp}_{L'}(s_{0}, \ldots, s_{n-1}) = \text{qftp}_{L'}(t_{0}, \ldots, t_{n-1}),
$$
then we have
$$
\text{tp}(a_{s_{0}},\ldots, a_{s_{n-1}}/A) = \text{tp}(a_{t_{0}},\ldots, a_{t_{n-1}}/A).
$$
\item In the case that $L' = L_{s,\alpha}$ for some $\alpha$, we say that an $I$-indexed indiscernible is $\emph{s-indiscernible}$.  As the only $L_{s,\alpha}$-structures we will consider will be trees, we will often refer to $I$-indexed indiscernibles in this case as \emph{s-indiscernible trees}.  
\item We say that $I$-indexed indiscernibles have the \emph{modeling property} if, given any $(a_{i} : i \in I)$ from $\mathbb{M}$ and set of parameters $A$, there is an \(I\)-indexed indiscernible \((b_{i} : i \in I)\) in $\mathbb{M}$ \emph{locally based} on \((a_{i} : i \in I)$ over $A$ -- i.e., given any finite set of formulas \(\Delta\) from \(L(A)\) and a finite tuple \((t_{0}, \ldots, t_{n-1})\) from \(I\), there is a tuple \((s_{0}, \ldots, s_{n-1})\) from \(I\) such that 
\[
\text{qftp}_{L'} (t_{0}, \ldots, t_{n-1}) =\text{qftp}_{L'}(s_{0}, \ldots , s_{n-1})
\]
and also 
\[
\text{tp}_{\Delta}(b_{t_{0}}, \ldots, b_{t_{n-1}}) = \text{tp}_{\Delta}(a_{s_{0}}, \ldots, a_{s_{n-1}}).
\]

\end{enumerate}
\end{defn}

Recall that, given a set $X$, we write $[X]^{<\omega}$ to denote the set of finite subsets of $X$.  

\begin{defn} \label{spead out def}
Suppose $(a_{\eta})_{\eta \in \mathcal{T}_{\alpha}}$ is a tree of tuples, and $C$ is a set of parameters.  
\begin{enumerate}
\item We say that $(a_{\eta})_{\eta \in \mathcal{T}_{\alpha}}$ is \emph{spread out over} C if for all $\eta \in \mathcal{T}_{\alpha}$ with $\text{dom}(\eta) =[\beta+1,\alpha)$ for some $\beta \in \alpha$, there is a global $C$-invariant type $q_{\eta} \supseteq \text{tp}(a_{\unrhd \eta \frown \langle 0 \rangle}/C)$ such that $(a_{\unrhd \eta \frown \langle i \rangle})_{i < \omega}$ is a Morley sequence over $C$ in $q_{\eta}$.  
\item Suppose $(a_{\eta})_{\eta \in \mathcal{T}_{\alpha}}$ is a tree which is spread out and $s$-indiscernible over $C$ and for all $w,v \in [\alpha \setminus \text{lim}(\alpha)]^{<\omega}$ with $|w| = |v|$,
$$
(a_{\eta})_{\eta \in \mathcal{T}_{\alpha} \upharpoonright w} \equiv_{C} (a_{\eta})_{\eta \in \mathcal{T}_{\alpha} \upharpoonright v}
$$
then we say that $(a_{\eta})_{\eta \in \mathcal{T}_{\alpha}}$ is a \emph{Morley tree} over $C$.  
\item A \emph{tree Morley sequence} over $C$ is a $C$-indiscernible sequence of the form $(a_{\zeta_{\beta}})_{\beta \in \alpha \setminus \text{lim}(\alpha)}$ for some Morley tree $(a_{\eta})_{\eta \in \mathcal{T}_{\alpha}}$ over $C$.  
\end{enumerate}
\end{defn}

\begin{rem}
Note that in Definition \ref{spread out def}(1), it is possible that $\mathrm{dom}(\eta) = [\beta + 1,\alpha)$ for a limit ordinal $\beta \in \alpha$, in which case $\eta \frown \langle i \rangle$, defined to be the function $\eta \cup \{(\beta, i)\}$, is not an element of $\mathcal{T}_{\alpha}$.  Nonetheless, the tuple $a_{\unrhd \eta \frown \langle i \rangle}$ still makes sense as the tuple whose elements are indexed by functions in the tree $\mathcal{T}_{\alpha}$ containing $\eta \frown \langle i \rangle$.  See the remarks after Definition \ref{concatenation}.  
\end{rem}

\begin{fact} \label{modeling}
\text{ }
\begin{enumerate}
\item For any $\alpha$, $\mathcal{T}_{\alpha}$-indexed indiscernibles have the modeling property \cite[Theorem 4.3]{KimKimScow} \cite[Corollary 5.6]{kaplan2017kim}.
\item Given a model $M \models T$, there is a cardinal $\kappa$ such that if $(a_{\eta})_{\eta \in \mathcal{T}_{\kappa}}$ is a tree of tuples, spread out and $s$-indiscernible over $M$, then there is a Morley tree $(b_{\eta})_{\eta \in \mathcal{T}_{\omega}}$ such that for all $w \in [\omega]^{<\omega}$, 
$$
(a_{\eta})_{\eta \in \mathcal{T}_{\kappa} \upharpoonright v} \equiv_{M} (b_{\eta})_{\eta \in \mathcal{T}_{\omega} \upharpoonright w}.  
$$
for some $v \in [\kappa \setminus \text{lim}(\kappa)]^{<\omega}$ \cite[Lemma 5.10]{kaplan2017kim}.
\end{enumerate}
\end{fact}

The interest in tree Morley sequences is that the genericity condition is sufficiently weak that they exist under broader hypotheses than invariant Morley sequences, yet is sufficiently strong to witness Kim-independence.  This is made precise below:

\begin{defn}
Suppose $M$ is a model and $(a_{i})_{i < \omega}$ is an $M$-indiscernible sequence.
\begin{enumerate}
\item Say $(a_{i})_{i < \omega}$ is a \emph{witness} for Kim-dividing over $M$ if, for all formulas $\varphi(x;a_{0})$ that Kim-divide over $M$, $\{\varphi(x;a_{i}) : i <\omega\}$ is inconsistent.
\item Say $(a_{i})_{i < \omega}$ is a \emph{strong witness} to Kim-dividing over $M$ if, for all $n$, the sequence $\langle (a_{n \cdot i}, a_{n \cdot i + 1}, \ldots, a_{n \cdot i + n-1}) : i < \omega \rangle$ is a witness to Kim-dividing over $M$.  
\end{enumerate}
\end{defn}

\begin{fact} \label{witnessfacts}   \cite[Proposition 7.9]{kaplan2017kim}
Suppose $T$ is NSOP$_{1}$ and $M \models T$.  
\begin{enumerate}
\item (Kim's Lemma for tree Morley sequences) $\varphi(x;a)$ Kim-divides over $M$ if and only if $\{\varphi(x;a_{i}) : i < \omega\}$ is inconsistent for some tree Morley sequence $(a_{i})_{i < \omega}$ over $M$ with $a_{0} = a$ if and only if $\{\varphi(x;a_{i}) : i < \omega\}$ is inconsistent for all tree Morley sequences $(a_{i})_{i < \omega}$ over $M$ with $a_{0} = a$.  \cite[Corollary 5.14]{kaplan2017kim}
\item The sequence $(a_{i})_{i < \omega}$ is a strong witness for Kim-dividing over $M$ if and only if $(a_{i})_{i < \omega}$ is a tree Morley sequence over $M$.  \cite[Proposition 7.9]{kaplan2017kim}
\end{enumerate}
\end{fact}

\begin{rem} \label{coheirs are strong witnesses}
The argument for \cite[Corollary 7.10]{kaplan2017kim} contains a proof that $\ind^{f}$-Morley sequences over models are strong witnesses to Kim-dividing.  Note that it follows, then, that coheir sequences over models are also strong witnesses to Kim-dividing, as $a \ind^{u}_{M} b$ implies $a \ind^{f}_{M} b$ for all $M \models T$.  
\end{rem}

Finally, we define one more kind of $\ind^{K}$-Morley sequence:

\begin{defn}
Suppose $M \models T$.  A \emph{total }$\ind^{K}$\emph{-Morley sequence over}$M$ is an $M$-indiscernible sequence $\langle a_{i} : i < \omega \rangle$ such that $a_{>i} \ind^{K}_{M} a_{\leq i}$ for all $i < \omega$.  
\end{defn}

\begin{fact} \label{sequence implications}
Suppose $T$ is NSOP$_{1}$ and $M \models T$.  
\begin{enumerate}
\item If $I$ is a tree Morley sequence over $M$, then $I$ is a total $\ind^{K}$-Morley sequence over $M$. 
\item If $I$ is a total $\ind^{K}$-Morley sequence over $M$, then $I$ is $\ind^{K}$-Morley over $M$.  
\end{enumerate}
\end{fact}

\begin{proof}
(2) is obvious so we prove (1).  Suppose $I = \langle a_{i} : i < \omega \rangle$ is a tree Morley sequence over $M$.  Let $(a_{\eta})_{\eta \in \mathcal{T}_{\omega}}$ be a Morley tree over $M$ with $a_{i} = a_{\zeta_{i}}$ for all $i < \omega$.  Then for all $i < \omega$, we have that $\langle a_{\unrhd \zeta_{i+1} \frown \langle j \rangle} : j < \omega \rangle$ is a Morley sequence over $M$ in a global $M$-invariant type which is $Ma_{\zeta_{\geq i+1}}$-indiscernible.  Therefore $a_{\zeta_{\geq i+1}} \ind^{K}_{M} a_{\unrhd \zeta_{i+1} \frown \langle 0 \rangle}$, which, in particular, implies $a_{> i} \ind^{K}_{M} a_{\leq i}$ for all $i < \omega$.  
\end{proof}

\section{Transitivity holds in NSOP$_{1}$ theories}

In this section, we prove the transitivity of Kim-independence in NSOP$_{1}$ theories.  The argument proceeds via an analysis of situations under which one can obtain sequences that are generic over more than one base simultaneously.  The heart of the argument is Proposition \ref{goodseq}, which proves the existence of a sequence that is a tree Morley sequence over a model and $\ind^{K}$-Morley over an elementary extension.  This, combined with symmetry, gives transitivity as an immediate consequence.  

Producing a sequence which is $\ind^{K}$-Morley over a model and a tree Morley sequence over an elementary extension is less involved.  The following lemma was implicit in \cite[Lemma 3.6]{24-Kaplan2017}:

\begin{lem} \label{kms over M tms over N}
Suppose $T$ is NSOP$_{1}$, $M \prec N \models T$, and $a \ind^{K}_{M} N$.  Then there is a tree Morley sequence $(a_{i})_{i < \omega}$ over $N$ with $a_{0} = a$ such that $a_{i} \ind^{K}_{M} Na_{<i}$ for all $i < \omega$.  In particular, $\langle a_{i} : i < \omega \rangle$ is simultaneously an $\ind^{K}$-Morley sequence over $M$ and a tree Morley sequence over $N$.  
\end{lem}

\begin{proof}
Let $\langle a_{i} : i \in \mathbb{Z}\rangle$ be a coheir sequence over $N$ with $a_{0} = a$.  Since coheir sequences are strong witnesses to Kim-dividing, by Remark \ref{coheirs are strong witnesses}, reversing the order of a sequence does not change whether or not it is a strong witness, and the fact that strong witnesses are tree Morley by Fact \ref{witnessfacts}(2), it follows that, setting $b_{i} = a_{-i}$, we have that $\langle b_{i} : i < \omega \rangle$ is a tree Morley sequence over $N$ with $b_{0} = a$.  

We claim this sequence also satisfies $b_{i} \ind^{K}_{M} Nb_{<i}$:  if not, then by symmetry, there is some $i$ such that $Nb_{<i} \nind^{K}_{M} b_{i}$ and this is witnessed by some $\varphi(x,n;b_{i}) \in \text{tp}(b_{<i}N/Mb_{i})$.  Because $\langle a_{i} : i \in \mathbb{Z} \rangle$ is a coheir sequence over $N$, we have, in particular, that $a_{>i} \ind^{u}_{N} a_{i}$.  Hence $b_{<i} \ind^{u}_{N} b_{i}$ so there must be some $n' \in N$ with $\models \varphi(n',n;b_{i})$.  But then $N \nind^{K}_{M} b_{i}$.  By symmetry and invariance, this contradicts $a \ind^{K}_{M} N$, since $b_{i} \equiv_{N} a$.  
\end{proof}

\begin{lem} \label{weirdextension}
Supose $T$ is NSOP$_{1}$ and $M \prec N \models T$.  If $b \ind^{K}_{M} N$ and $c \ind^{K}_{M} N$, then there is $c' \equiv_{N} c$ such that $bc' \ind^{K}_{M} N$ and $b \ind^{K}_{N} c'$.  
\end{lem}

\begin{proof}
Define a partial type $\Gamma(x;N,b)$ by 
$$
\Gamma(x;N,b) = \text{tp}(c/N) \cup \{\neg \varphi(x,b;n) : \varphi(x,y;n) \text{ Kim-divides over }M\}.
$$
By Lemma \ref{kms over M tms over N}, we may construct an $N$-indiscernible sequence $\langle b_{i} : i < \omega \rangle$ such that $b_{0} = b$, $b_{i+1} \ind^{K}_{M} Nb_{\leq i}$, and $\langle b_{i} : i < \omega \rangle$ is a tree Morley sequence over $N$.

\textbf{Claim 1}:  $\bigcup_{i < \omega} \Gamma(x;N,b_{i})$ is consistent. 

\emph{Proof of claim}:  By induction on $n$, we will choose $c_{n} \ind^{K}_{M} Nb_{< n}$ such that 
$$
c_{n} \models \bigcup_{i < n} \Gamma(x;N,b_{i}).
$$
For $n = 0$, we may set $c_{0} = c$ and the condition is satisfied since $c \ind^{K}_{M} N$.  

Suppose we are given $c_{n} \ind^{K}_{M} Nb_{< n}$ realizing $\bigcup_{i < n} \Gamma(x;N,b_{i})$.  By extension, choose $c' \equiv_{M} c$ with $c' \ind^{K}_{M} b_{n}$.  As $b_{n} \ind^{K}_{M} Nb_{< n}$, we may apply the strengthened independence theorem, Fact \ref{fact:Kim Morley is consistent}(3), to find $c_{n+1} \models \text{tp}(c_{n}/Nb_{< n}) \cup \text{tp}(c'/Mb_{n})$ with $c_{n+1} \ind^{K}_{M} Nb_{< n+1}$ and $b_{n}c_{n+1} \ind^{K}_{M} Nb_{< n}$.  In particular, $b_{n}c_{n+1} \ind^{K}_{M} N$, so $c_{n+1} \models \Gamma(x;N,b_{n})$.  This gives $c_{n+1} \models \bigcup_{i < n+1} \Gamma(x;N,b_{i})$.  The claim follows by compactness.\qed

Now define a partial type $\Delta(x;N,b)$ by 
\begin{eqnarray*}
\Delta(x;N,b) &=& \Gamma(x;N,b)\\
& \cup& \{ \neg \psi(x;b): \psi(x;b) \in L(Nb)\text{ Kim-divides over }N\}.
\end{eqnarray*}
\textbf{Claim 2}:  $\Delta(x;a,b)$ is consistent.  

\emph{Proof of claim}:  Suppose not.  Then, by definition of $\Delta(x;N,b)$, compactness, and the equality of Kim-forking and Kim-dividing, we have
$$
\Gamma(x;N,b) \vdash \psi(x;b),
$$
for some $\psi(x;b) \in L(Nb)$ that Kim-divides over $N$.  Then we have 
$$
\bigcup_{i < \omega} \Gamma(x;N,b_{i}) \vdash \{\psi(x;b_{i}) : i < \omega\}.
$$
The left-hand side is consistent by Claim 1 but the right hand side is inconsistent by Kim's lemma and the choice of $\overline{b}$, a contradiction that proves the claim.
\qed

Now let $c' \models \Delta(x;N,b)$.  Then, by symmetry, we have $c' \equiv_{N} c$, $bc' \ind^{K}_{M} N$, and $b \ind^{K}_{N} c'$ which is what we want.   
\end{proof}

\begin{prop}\label{goodseq}
Suppose $T$ is NSOP$_{1}$ and $M \prec N \models T$.  If $a \ind^{K}_{M} N$, then there is a sequence $\langle a_{i} : i < \omega \rangle$ with $a_{0} = a$ which is a tree Morley sequence over $M$ and Kim-Morley over $N$. 
\end{prop}

\begin{proof}
By induction on $\alpha$, we will construct trees $(a^{\alpha}_{\eta})_{\eta \in \mathcal{T}_{\alpha}}$ satisfying the following conditions:
\begin{enumerate}
\item For all $\eta \in \mathcal{T}_{\alpha}$, $a^{\alpha}_{\eta} \models \text{tp}(a/N)$.  
\item $(a^{\alpha}_{\eta})_{\eta \in \mathcal{T}_{\alpha}}$ is $s$-indiscernible over $N$, spread out over $M$.
\item If $\alpha$ is a successor, then $a^{\alpha}_{\emptyset} \ind^{K}_{N} (a^{\alpha}_{\eta})_{\eta \in \mathcal{T}_{\alpha}}$.
\item $(a^{\alpha}_{\eta})_{\eta \in \mathcal{T}_{\alpha}} \ind^{K}_{M} N$.  
\item If $\alpha < \beta$, then for all $\eta \in \mathcal{T}_{\alpha}$, $a^{\beta}_{\iota_{\alpha \beta}(\eta)} = a^{\alpha}_{\eta}$.  
\end{enumerate}
To begin, put $a^{0}_{\emptyset} = a$.  At a limit stage $\delta$, we define $(a^{\delta}_{\eta})_{\eta \in \mathcal{T}_{\delta}}$ by $a^{\delta}_{\iota_{\alpha \delta}(\eta)} = a^{\alpha}_{\eta}$ for all $\alpha < \delta$ and $\eta \in \mathcal{T}_{\alpha}$.  This is well-defined by (5) and the definition of $\mathcal{T}_{\delta}$.  Moreover, it clearly satisfies (1), (2) is trivial, and (3) and (4) are satisfied by finite character.  

Now in the successor stage, we will construct $(a^{\alpha+1}_{\eta})_{\eta \in \mathcal{T}_{\alpha+1}}$.  Let $\overline{b} = \langle (a^{\alpha}_{\eta,i})_{\eta \in \mathcal{T}_{\alpha}} : i < \omega \rangle$ be a coheir sequence over $M$ with $a^{\alpha}_{\eta,0} = a^{\alpha}_{\eta}$ for all $\eta \in \mathcal{T}_{\alpha}$.  By (4), we may assume $\overline{b}$ is $N$-indiscernible and $\overline{b} \ind^{K}_{M} N$ by the chain condition (Lemma \ref{chain condition for indk}).  Apply Lemma \ref{weirdextension} to get $b \equiv_{N} a$ such that $b \ind^{K}_{N} \overline{b}$ and $b\overline{b} \ind^{K}_{M} N$.  Define a tree $(b_{\eta})_{\eta \in \mathcal{T}_{\alpha+1}}$ by $b_{\emptyset} = b$ and $b_{\langle i \rangle \frown \eta} = a^{\alpha}_{\eta,i}$ for all $i < \omega$, $\eta \in \mathcal{T}_{\alpha}$.  Now let $(a^{\alpha+1}_{\eta})_{\eta \in \mathcal{T}_{\alpha+1}}$ be an $s$-indiscernible tree over $N$, locally based on the tree $(b_{\eta})_{\eta \in \mathcal{T}_{\alpha}}$ over $N$.  By an automorphism, we may assume that $a^{\alpha+1}_{\langle 0 \rangle \frown \eta} = a^{\alpha}_{\eta}$ for all $\eta \in \mathcal{T}_{\alpha}$, so (5) is satisfied.  By construction and induction, $(b_{\eta})_{\eta \in \mathcal{T}_{\alpha+1}}$ is spread out over $N$ and $b_{\eta} \models \text{tp}(a/N)$ for all $\eta \in \mathcal{T}_{\alpha+1}$ so $(a^{\alpha+1}_{\eta})_{\eta \in \mathcal{T}_{\alpha+1}}$ satisfies (1) and (2).  Likewise, since $b \ind^{K}_{N} \overline{b}$ and $(b_{\eta})_{\eta \in \mathcal{T}_{\alpha+1}} \ind^{K}_{M} N$ by construction, and because of the fact that Kim-forking is witnessed by formulas, it follows from the fact that $(a^{\alpha+1}_{\eta})_{\eta \in \mathcal{T}_{\alpha+1}}$ is locally based on $(b_{\eta})_{\eta \in \mathcal{T}_{\alpha+1}}$ over $N$ that $a^{\alpha+1}_{\emptyset} \ind^{K}_{N} a^{\alpha+1}_{\vartriangleright \emptyset}$ and $(a^{\alpha+1}_{\eta})_{\eta \in \mathcal{T}_{\alpha+1}} \ind^{K}_{M} N$ as well, so (3) and (4) are satisfied.  This completes the construction.  

Let $(a_{\eta})_{\eta \in \mathcal{T}_{\omega}}$ be the tree obtained by applying Fact \ref{modeling}.  Then $(a_{\zeta_{\alpha}})_{\alpha < \omega}$ is the desired sequence.    
\end{proof}

\begin{thm} \label{transitivity theorem}
Suppose $T$ is NSOP$_{1}$, $M \prec N \models T$, $a \ind^{K}_{M} N$ and $a \ind^{K}_{N}b$.  Then $a \ind^{K}_{M} Nb$.  
\end{thm}

\begin{proof}
Suppose $a,b,M,$ and $N$ are given as in the statement.  By Proposition \ref{goodseq}, there is a sequence $I = \langle a_{i} : i < \omega \rangle$ with $a_{0} = a$ such that $I$ is a tree Morley sequence over $M$ and an $\ind^{K}$-Morley sequence over $N$.  Since $b \ind^{K}_{N} a$, there is $I' \equiv_{Na} I$ such that $I'$ is $Nb$-indiscernible, by compactness, Ramsey, and Fact \ref{witnessfacts}.  Then $I'$ still a tree Morley sequence over $M$ so, by Kim's lemma for tree Morley sequences, $Nb \ind^{K}_{M} a$, so we may conclude by symmetry.
\end{proof}

Transitivity allows one to easily obtain analogues for Kim-independence of the ``algebraically reasonable" properties of Kim- and algebraic-independence proved in \cite{kruckman2018generic}.  For example, the following is the analogue of ``algebraically reasonable extension" \cite[Theorem 2.15]{kruckman2018generic}:  

\begin{cor}
Suppose $T$ is NSOP$_{1}$, $M \prec N \models T$, and $a \ind^{K}_{M} N$.  Then given any $b$, there is $a'$ with $a' \equiv_{N} a$, $a' \ind^{K}_{M} Nb$, and $a' \ind^{K}_{N} b$.  
\end{cor}

\begin{proof}
Applying extension, we obtain $a' \equiv_{N} a$ so that $a' \ind^{K}_{N} b$.  By invariance, $a' \ind^{K}_{M} N$ so by transitivity, $a' \ind^{K}_{M} Nb$.  
\end{proof}

%\begin{proof}
%Towards contradiction, suppose $M \prec N \models T$, $a \ind^{K}_{M} N$, and $b$ is a tuple such that there is no $a'$ as in the statement.  Then let $\Gamma(x;N,b)$ be the following partial type:
%$$
%\text{tp}(a/N) \cup \{\neg \varphi(x;n,b) : \varphi(x;n,b) \in L(Nb) \text{ and }\varphi(x;n,b) \text{ Kim-divides over }M\}.
%$$
%By assumption and the fact that Kim-forking equals Kim-dividing in an NSOP$_{1}$ theory, we have 
%$$
%\Gamma(x;N,b) \vdash \psi(x;b)
%$$
%for some $\psi(x;b) \in L(Nb)$ that Kim-divides over $N$.  Let $\overline{b} = (b_{i})_{i < \omega}$ be a coheir sequence over $N$ with $b_{0} = b$.  Then we have 
%$$
%\bigcup_{i < \omega} \Gamma(x;N,b_{i}) \vdash \{\psi(x;b_{i}) : i < \omega\},
%$$
%and $\{\psi(x;b_{i}) : i < \omega\}$ is inconsistent by Kim's lemma, so $\bigcup_{i < \omega} \Gamma(x;N,b_{i})$ is inconsistent.  But by extension, we may choose $a' \equiv_{N} a$ with $a' \ind^{K}_{M} N\overline{b}$.  By definition of $\Gamma$, $a' \models \bigcup_{i < \omega} \Gamma(x;N,b_{i})$, a contradiction.  This completes the proof.  
%\end{proof}

\subsection{An example}

In this subsection, we present an example that illustrates two important phenomena simultaneously.  First, it shows that if $T$ is NSOP$_{1}$, $M \prec N \models T$ and $a \ind^{K}_{M} N$, then it is not necessarily possible to find $I = \langle a_{i} : i < \omega \rangle$ that is a coheir sequence over $N$ with $a_{0} = a$ and $I \ind^{K}_{M} N$.  Our example shows that Lemma \ref{kms over M tms over N}\textemdash this lemma shows that, in this situation, a coheir sequence starting with $a$ over $N$ is $\ind^{K}$-Morley over $M$\textemdash is optimal, as, in general, one cannot improve it to conclude that a coheir sequence over $N$ is a stronger form of Morley sequence over $M$.  In particular, it is not the case that every tree Morley sequence over $N$ is automatically a tree Morley sequence over $M$, as one might hope, as we produce a coheir sequence (which is therefore tree Morley over $N$) which is not tree Morley over $M$.   Secondly, our example shows that it is possible, in an NSOP$_{1}$ theory, that there is an $\ind^{K}$-Morley sequence that is neither a tree Morley sequence nor a total $\ind^{K}$-Morley sequence (later, in Corollary \ref{tms = tree}, we will show the notions of tree Morley sequence and total $\ind^{K}$-Morley sequence are equivalent).  In particular, we show that there is a sequence $\langle a_{i} : i < \omega \rangle$ that is $\ind^{K}$-Morley over $M$ but $a_{2}a_{3} \nind^{K}_{M} a_{0}a_{1}$.  

\begin{fact} \label{winkler facts}
Let $L$ be the language consisting of a single binary function $f$.
\begin{enumerate}
\item The empty $L$-theory has a model completion $T_{f}$, which eliminates quantifiers. \cite{winkler1975model} \cite[Corollary 3.10]{kruckman2018generic}
\item Modulo $T_{f}$, for all sets $A$, $\text{acl}(A) = \text{dcl}(A) = \langle A \rangle$, where $\langle A \rangle$ denotes the substructure generated by $A$.  \cite[Corollary 3.11]{kruckman2018generic}
\item $T_{f}$ is NSOP$_{1}$ and Kim-independence coincides with algebraic independence:  for any tuples $a,b$, if $M \models T_{f}$, $a \ind^{K}_{M} b$ if and only if $\langle aM \rangle \cap \langle bM \rangle = M$. \cite[Corollary 3.13]{kruckman2018generic}
\end{enumerate}
\end{fact}

%\begin{rem}
%It follows from Fact \ref{winkler facts} (1) and (2) that, in $T_{f}$, that all $M$-definable functions are given piecewise by terms:  if $f : X \to Y$ is an $M$-definable function, let $\Gamma$ consist of all complete types over $M$ with $p \vdash X$.  Then if $a \models p$, since $f(a) \in \text{dcl}(aM) = \langle aM \rangle$, there is a term $t_{p}(x;m_{p})$, possibly with parameters from $M$, so that $f(a) = t_{p}(a;m_{p})$.  Then $p \vdash f(x) = t_{p}(x;m_{p})$.  Since $\Gamma$ is closed, we know by compactness that only finitely many of the terms $t_{p}$ are needed to define $f$, so $f$ is piecewise definable by terms.  
%\end{rem}

\begin{exmp}
Let $M$ be a countable model of $T_{f}$ and $N$ an $\aleph_{1}$-saturated elementary extension, all contained in the monster model $\mathbb{M} \models T_{f}$.  Pick elements $m_{*} \in M$ and $n_{*} \in N \setminus M$.  In $N \setminus M$, we can find a countable set of distinct elements $B = \{b_{i} : i < \omega\}$ such that $f|_{(B \times M) \cup (M \times B) \cup \Delta} = \{m_{*}\}$ and $f|_{(B \times B) \setminus \Delta} = \{n_{*}\}$, where $\Delta = \{(b_{i},b_{i}) : i < \omega\}$.  Let $\mathcal{D}$ be a non-principal ultrafilter on $N$ concentrating on $B$ and $q = \text{Av}(\mathcal{D},\mathbb{M})$.  Let $I = (a_{i})_{i < \omega} \models q^{\otimes \omega}|_{N}$ be a Morley sequence in $q$ over $N$.  

We claim $a_{0} \ind^{K}_{M} N$.  This is equivalent to the assertion that $\langle a_{0}M \rangle \cap N = M$.
Suppose $c \in \langle a_{0}M \rangle \cap N$.  Then there is a term $t(x;m)$, possibly with parameters from $M$, such that $t(a_{0};m) = c$ and therefore $\{i : t(b_{i};m) = c\} \in \mathcal{D}$.  One may easily check that if $s$ is a constant term in the language $L(Mb_{i})$, i.e. $L$ with constants for $M$ and $b_{i}$, then either $s = b_{i}$ or there is $m' \in M$ with $s = m'$.  This is clear for the constants and, since $f(n,b_{i}) = f(b_{i},n) = f(b_{i},b_{i}) = m_{*}$ for all $n \in M$, the induction follows.  Since the $b_{i}$ are pairwise distinct and $\{i : t(b_{i};m) = c\} \in \mathcal{D}$, it is clear that $c$ is not equal to any $b_{i}$, so it follows that $c \in M$.  

However, $f(b_{i},b_{j}) = n_{*} \in N\setminus M$ for all $i,j < \omega$ so $f(a_{0},a_{1}) = n_{*}$.  Therefore $\text{dcl}(a_{0},a_{1}M) \cap N \neq M$, which shows $a_{0}a_{1} \nind^{K}_{M} N$.  This shows in particular $I \nind^{K}_{M} N$.  

Next, in the proof of Lemma \ref{kms over M tms over N}, we show that if $T$ is NSOP$_{1}$, $M \prec N$ and $b \ind^{K}_{M} N$, then for any coheir sequence $\langle b_{i} : i < \omega \rangle$ in $\text{tp}(b/N)$, we have $b_{>i} \ind^{K}_{M} b_{i} $ for all $i < \omega$.  It follows that $a_{>i} \ind^{K}_{M} a_{i}$ and thus $I$ is an $\ind^{K}$-Morley sequence over $M$ indexed in reverse.  However, we have $f(a_{2},a_{3}) = f(a_{0},a_{1}) = n_{*} \not\in M$ so $a_{0}a_{1} \nind^{K}_{M} a_{2}a_{3}$, which shows that $I$ is not a total $\ind^{K}$-Morley sequence.  
\end{exmp}

\section{Transitivity implies NSOP$_{1}$}

In this section, we complete the characterization of NSOP$_{1}$ theories by the transitivity of Kim-independence.  The argument is loosely inspired by the proof due to Kim that transitivity of non-forking independence implies simplicity \cite[Theorem 2.4]{kim2001simplicity}.  However, we have to deal with a more complicated combinatorial configuration as well as the need to produce \emph{models} over which we may observe a failure of transitivity from SOP$_{1}$.  We begin by observing a combinatorial consequence of SOP$_{1}$ arising from the witnessing array of pairs and then work in a Skolemization of a given SOP$_{1}$ theory to find the desired counter-example to transitivity.  

\begin{lem}\label{goodindiscernible}
Suppose $T$ has SOP$_{1}$.  Then there is a formula $\varphi(x;y)$ and an indiscernible sequence $(a_{i},c_{i,0},c_{i,1})_{i < \omega}$ such that 
\begin{enumerate}
\item For all $i < \omega$, $a_{i} \models \{\varphi(x;c_{j,0}) : j \leq i\}$.
\item $\{\varphi(x;c_{i,1}) : i < \omega\}$ is $2$-inconsistent.
\item For all $i < \omega$, $c_{i,0} \equiv_{a_{<i},c_{<i,0}c_{<i,1}} c_{i,1}$.
\end{enumerate}
\end{lem}

\begin{proof}
Let $\kappa$ be a cardinal sufficiently large relative $|M|$ to apply Fact \ref{modeling}.  Because $T$ has SOP$_{1}$, we know by Fact \ref{karyversion} and compactness, there is a formula $\varphi(x;y)$ and an indiscernible sequence $(c_{i,0},c_{i,1})_{i < \kappa}$ such that 
\begin{itemize}
\item $\{\varphi(x;c_{i,0}) : i < \kappa\}$ is consistent.
\item $\{\varphi(x;c_{i,1}) : i < \kappa\}$ is $2$-inconsistent.
\item For all $i < \kappa$, $c_{i,0} \equiv_{c_{<i,0}c_{<i,1}} c_{i,1}$ and $c_{n,0} \equiv_{a_{<n}c_{<n,0}c_{<n,1}} c_{n,1}$.  
\end{itemize}
By induction on $n < \omega$, we will build $(a_{i})_{i < n}$ and $(c^{n}_{i,0},c^{n}_{i,1})_{i < \kappa}$ such that, for all $n < \omega$, 
\begin{enumerate}
\item $\{\varphi(x;c^{n}_{i,0}) : i < \kappa\}$ is consistent.
\item $\{\varphi(x;c^{n}_{i,1}) : i < \kappa\}$ is $2$-inconsistent.
\item For all $i < \kappa$, $c^{n}_{i,0} \equiv_{c^{n}_{<i,0}c^{n}_{<i,1}} c^{n}_{i,1}$ and $c^{n}_{n,0} \equiv_{a_{<n}c^{n}_{<n,0}c^{n}_{<n,1}} c^{n}_{n,1}$.  
\item For all $i < n$, $a_{i} \models \{\varphi(x;c^{n}_{j,0}) : j \leq i\}$.
\item For $m \leq n$, $(c^{m}_{m,0},c^{m}_{m,1}) = (c^{n}_{m,0},c^{n}_{m,1})$.  
\end{enumerate}
To begin, we define $(c^{0}_{i,0},c^{0}_{i,1})_{i < \kappa}$ by setting $(c^{0}_{i,0},c^{0}_{i,1}) = (c_{i,0}, c_{i,1})$ for all $i < \kappa$.  This, together with the empty sequence of $a_{i}$'s satisfies (1)\textemdash(3). For $n=0$, (4) and (5) are vacuous, so this handles the base case.

Now suppose for $n$, we have constructed $(a_{i})_{i < n}$, $(c^{n}_{i,0},c^{n}_{i,1})_{i < \kappa}$.  Choose $a_{n} \models \{\varphi(x;c^{n}_{i,0}) : i \leq n\}$.  Now by the pigeonhole principle, there are $i_{*} > j_{*} > n$ such that $c^{n}_{i_{*},1} \equiv_{a_{\leq n}c^{n}_{\leq n,0}c^{n}_{\leq n,1}} c^{n}_{j_{*},1}$.  As $c^{n}_{i_{*},0} \equiv_{c^{n}_{<i_{*},0}c^{n}_{<i_{*},1}} c^{n}_{i_{*},1}$, there is $\sigma \in \text{Aut}(\mathbb{M}/c^{n}_{<i_{*},0}c^{n}_{<i_{*},1})$ with $\sigma(c^{n}_{i_{*},0}) = c^{n}_{i_{*},1}$.  Define a new array by setting $(c^{n+1}_{m,0},c^{n+1}_{m,1}) = (c^{n}_{m,0},c^{n}_{m,1})$ for all $m \leq n$, $(c^{n+1}_{n+1,0},c^{n+1}_{n+1,1}) = (c^{n}_{i_{*},1},c^{n}_{j_{*},1})$, and finally $(c^{n+1}_{n+\alpha,0},c^{n+1}_{n+\alpha,1}) = \sigma(c^{n}_{i_{*}+\alpha,0},c^{n}_{i_{*}+\alpha,1})$ for all $2 \leq \alpha < \kappa$.  

Now we check that this satisfies the requirements.  For (1), note that $\{\varphi(x;c^{n+1}_{i,0}) : i < \kappa\}$ is equal to $\{\varphi(x;\sigma(c^{n}_{i,0})) : i \leq n \text{ or }i \geq i_{*}\}$ and this is consistent because $\{\varphi(x;c^{n}_{i,0}) : i \leq n \text{ or } i \geq i_{*}\}$ is consistent and $\sigma$ is an automorphism.  Likewise, $\{\varphi(x;c^{n+1}_{i,1}) : i < \kappa\}$ is equal to $\{\varphi(x;\sigma(c^{n}_{i,1})) : i \leq n, i = j_{*}, \text{ or }i > i_{*}\}$, so this is $2$-inconsistent because $\{\varphi(x;c^{n}_{i,1}) : i < \kappa\}$ is $2$-inconsistent and $\sigma$ is an automorphism.  (3)\textemdash(5) are immediate from our construction.  This completes the induction.

Now define $(c^{\omega}_{i,0},c^{\omega}_{i,1})_{i < \omega}$ such that $(c^{\omega}_{i,0},c^{\omega}_{i,1}) = (c^{j}_{i,0},c^{j}_{i,1})$ for some, equivalently all, $j \geq i$.  Then $(a_{i},c^{\omega}_{i,0},c^{\omega}_{i,1})_{i < \omega}$ satisfies conditions (1)\textemdash(3) so, after extracting an indiscernible sequence, we conclude.  
\end{proof}

\begin{rem}
If $\varphi(x;y)$ witnesses SOP$_{1}$ in $T$, it is clear from the definition that $\varphi$ will witness SOP$_{1}$ in any expansion $T'$ of $T$ and hence we may apply the above lemma to find $(a_{i},c_{i,0},c_{i,1})_{i < \omega}$ which are morever $L^{\text{Sk}}$-indiscernible and satisfy $c_{i,0} \equiv^{L^{Sk}}_{a_{<i},c_{<i,0}c_{<i,1}} c_{i,1}$ for all $i$ in $\mathbb{M}^{Sk}$, where the $L^{\text{Sk}}$-structure $\mathbb{M}^{\text{Sk}}$ is a monster model of an expansion of $T$ with Skolem functions.  See, e.g., \cite[Remark 2.5]{kaplan2017kim}.
\end{rem}

\begin{prop} \label{transitivity implies nsop1}
Suppose $T$ has SOP$_{1}$.  Then there are models $M \prec N \models T$ and tuples $a$ and $c$ such that $a \ind^{u}_{M} N$, $a \ind^{u}_{N}c$ and $a \nind^{K}_{M} Nc$.  
\end{prop}

\begin{proof}
Fix a Skolemization $T^{\text{Sk}}$ of $T$ in the language $L^{\text{Sk}}$ and work in a monster model $\mathbb{M}^{\text{Sk}} \models T^{\text{Sk}}$.  We will write $\equiv^{L^{\mathrm{Sk}}}$ to denote equality of types in the language $L^{\mathrm{Sk}}$ and $\equiv$ to denote equality of types in the language $L$.  By Lemma \ref{goodindiscernible} and compactness, we can find an $L$-formula $\varphi(x;y)$ and an $L^{\text{Sk}}$-indiscernible sequence $(a_{i},c_{i,0},c_{i,1})$ such that 
\begin{enumerate}
\item For all $i \in \mathbb{Q}$, $a_{i} \models \{\varphi(x;c_{j,0}) : j \leq i\}$.
\item $\{\varphi(x;c_{i,1}) : i \in \omega\}$ is $2$-inconsistent.
\item For all $i \in \mathbb{Q}$, $c_{i,0} \equiv^{L^{\text{Sk}}}_{a_{<i},c_{<i,0}c_{<i,1}} c_{i,1}$.
\end{enumerate}
Define $M = \text{Sk}(a_{<0}c_{<0,0}c_{<0,1})$ and $N = \text{Sk}(a_{<0},c_{<0,0},c_{<0,1},a_{>1})$.  Note that we have $M \prec N$.  In the claims below, independence is understood to mean independence with respect to the $L$-theory $T$.  

\textbf{Claim 1}:  $a_{1} \ind^{u}_{M} N$.

\emph{Proof of claim}:  Fix a formula $\psi(x;n) \in \text{tp}(a_{1}/N)$.  We can write the tuple $n = t(a,c)$ where $t$ is a tuple of Skolem terms, $a$ is a finite tuple from $a_{<0}a_{>1}$ and $c$ is a finite tuple from $c_{<0,0}c_{<0,1}$.  As $a$ and $c$ are finite, there is some rational $\epsilon <0$ such that $a$ and $c$ come from $a_{<\epsilon}a_{>1}$ and $c_{<\epsilon,0}c_{<\epsilon,1}$ respectively.  By indiscernibility, $\psi(x;n)$ is realized also by any $a_{\delta}$ with $\epsilon < \delta < 0$, which is in $M$.\qed

\textbf{Claim 2}:  $a_{1} \ind^{u}_{N} c_{0,0}$.

\emph{Proof of claim}:  This has a similar proof to Claim 1.  Given any $\psi(x;n,c_{0,0}) \in \text{tp}(a_{1}/Nc_{0,0})$, as before, we can write the tuple $n = t(a,c)$ where $t$ is a tuple of Skolem terms, $a$ is a finite tuple from $a_{<0}a_{>1}$ and $c$ is a finite tuple from $c_{<0,0}c_{<0,1}$.  Because these tuples are finite, there is a rational $\epsilon > 1$ such that $a$ comes from $a_{<0}a_{>\epsilon}$.  Then by indiscernibility, $\psi(x;n,c_{0,0})$ is satisfied by any $a_{\delta}$ with $1 < \delta < \epsilon$, all of which are in $N$. \qed

\textbf{Claim 3}:  $a_{1} \nind^{K}_{M} Nc_{0,0}$.

\emph{Proof of claim}:  We will show even $a_{1} \nind^{K}_{M} c_{0,0}$.  Let $\mathcal{D}$ be an ultrafilter on $M$ containing $\{c_{i,1} : i \in (\epsilon,0)\}$ for every $\epsilon < 0$.  By $L^{\mathrm{Sk}}$-indiscernibility, we have $c_{0,1} \models \text{Av}(\mathcal{D},M)$.  Then there is a sequence $(b_{i})_{i < \omega} \models \text{Av}(\mathcal{D}, \mathbb{M}^{\text{Sk}})^{\otimes \omega}|_{M}$ with $b_{0} = c_{0,1}$.  By (2) and the choice of $\mathcal{D}$, we know $\{\varphi(x;b_{i}) : i < \omega\}$ is $2$-inconsistent so $\varphi(x;c_{0,1})$ Kim-divides over $M$.  Moreover, $c_{0,0} \equiv^{L^{\text{Sk}}}_{a_{<0}c_{<0,0}c_{<0,1}} c_{0,1}$ so, in particular, $c_{0,0} \equiv_{M} c_{0,1}$ from which it follows also that $\varphi(x;c_{0,0})$ Kim-divides over $M$.  By (1), we have $\models \varphi(a_{1},c_{0,0})$ so $a_{1} \nind^{K}_{M} c_{0,0}$.  \qed

The claims taken together show $a_{1} \ind^{n}_{M} N$, $a_{1} \ind^{u}_{N} c_{0,0}$, and $a_{1} \nind^{K}_{M} c_{0,0}$, which completes the proof.  
\end{proof}

\begin{cor}
The following are equivalent:
\begin{enumerate}
\item $T$ is NSOP$_{1}$.
\item $\ind^{K}$ satisfies the following weak form of transitivity:  if $M \prec N \models T$, $a \ind^{u}_{M} N$ and $a \ind^{u}_{N} b$, then $a \ind^{K}_{M} Nb$.  
\item $\ind^{K}$ satisfies transitivity:  if $M \prec N \models T$, $a \ind^{K}_{M} N$ and $a \ind^{K}_{N} b$, then $a \ind^{K}_{M} Nb$.  
\end{enumerate}
\end{cor}

\begin{proof}
Theorem \ref{transitivity theorem} establishes (1)$\implies$(3), Proposition \ref{transitivity implies nsop1} shows (2)$\implies$(1), and (3)$\implies$(2) is immediate from the fact that $\ind^{u}$ implies $\ind^{K}$.  
\end{proof}

\section{$\ind^{K}$-Morley sequences are witnesses}

In this section, we characterize NSOP$_{1}$ by the property that $\ind^{K}$-Morley sequences are witnesses to Kim-dividing.  The non-structure direction of this characterization was already observed in \cite[Theorem 3.16]{kaplan2017kim}: if $T$ has SOP$_{1}$ then $\ind^{K}$-Morley sequences will not always witness Kim-dividing.  The more interesting direction goes the other way, showing that in the NSOP$_{1}$ context, $\ind^{K}$-Morley sequences are witnesses.  This is a significant technical development in the study of NSOP$_{1}$ theories, as it, for example, obviates the need in many cases to construct tree Morley sequences.  We give some applications below.  

% if $T$ has SOP$_{1}$, then there is a formula $\varphi(x;a)$ that Kim-divides over $M$ and a coheir sequence $\langle a_{i} : i < \omega \rangle$ over $M$ with $a_{0} = a$ so that $\{\varphi(x;a_{i}) : i < \omega\}$ is consistent.  Since this coheir sequence must be $\ind^{K}$-Morley, this shows that 

\begin{thm} \label{witnessing theorem}
Suppose that $\varphi\left(x,a\right)$ Kim-divides over $M$. Suppose
that $\langle a_{i}: i<\omega\rangle$ is an $\ind^{K}$-Morley sequence over $M$,
starting with $a$. Then $\{\varphi\left(x,a_{i}\right):i<\omega\}$
is inconsistent.
\end{thm}

\begin{proof}
Suppose not. Let $\kappa=\left|M\right|+|T|$ and extend the sequence to have length $\kappa^{+}$.
It suffices to find an increasing continuous sequence of models $\langle N_{i}: i<\kappa^{+}\rangle$
such that $N_{i}$ contains $a_{<i}$, $\left|N_{i}\right|\leq \kappa$,
$N_{0}=M$ and such that $a_{i}\ind_{M}^{K}N_{i}$. To see this, suppose that $c\models\{\varphi\left(x,a_{i}\right):i<\kappa^{+}\}$.
Then by local character, Remark \ref{local character for bigger cardinals}, for some $i<\kappa^{+}$, $c\ind_{N_{i}}^{K}N_{\kappa^{+}}$
where $N_{\kappa^{+}}=\bigcup_{i<\kappa^{+}}N_{i}$, as $\{N_{i} : \kappa \leq i < \kappa^{+}\}$ is a club subset of $[N_{\kappa^{+}}]^{\kappa}$. Hence $c\ind_{N_{i}}^{K}a_{i}$.
However, $a_{i}\ind_{M}^{K}N_{i}$ and hence by transitivity and symmetry,
$N_{i}c\ind_{M}^{K}a_{i}$ contradicting our assumption that $\varphi\left(x,a\right)$
Kim-divides over $M$ and hence also $\varphi(x;a_{i})$, by invariance. 

\textbf{Claim}:  There is a partial type $\Gamma(\overline{x})$ over $a_{<\kappa^{+}}M$ such that:
\begin{enumerate}
\item We have $\overline{x}=\langle \overline{x}_{\alpha}: \alpha<\kappa^{+} \rangle$ is an increasing
continuous sequence of tuples of variables such that $\left|\overline{x}_{\alpha}\right|=\kappa$,
and such that $\overline{x}_{\alpha+1}$ contains $\kappa$
new variables not in $\overline{x}_{\alpha}$ for all $\alpha<\kappa^{+}$.
\item $\Gamma\left(\overline{x}\right)$ asserts that $\overline{x}_{\alpha}$ enumerates a model containing
$Ma_{<\alpha}$ for all $\alpha < \kappa^{+}$.
\end{enumerate}

\emph{Proof of claim}:  We define $\Gamma(\overline{x})$ as a continuous increasing union of partial types $\Gamma_{\alpha}(\overline{x}_{\alpha})$ for $\alpha < \kappa^{+}$.  Suppose we are given $\Gamma_{\delta}(\overline{x}_{\delta})$ for $\delta < \alpha$.  

If $\alpha = \beta+1$, then we define $\overline{x}_{\alpha,0} = \overline{x}_{\beta}$ and then, given $\overline{x}_{\alpha,i}$, we define $\Lambda_{i}$ to be the set of all partitioned formulas $\varphi(y;\overline{x})$ where the parameters of $\varphi$ come from $a_{<\alpha}M$ and the parameter variables $\overline{x}$ of $\varphi$ are among $\overline{x}_{\alpha,i}$.  Now define $\overline{x}_{\alpha,i+1} = \overline{x}_{\alpha,i}$ together with a new variable $x_{\lambda}$ for each $\lambda \in \Lambda_{i}$.  Finally $\overline{x}_{\alpha} = \bigcup_{i<\omega} \overline{x}_{\alpha,i}$.  Let $\Gamma_{\alpha,0} = \Gamma_{\beta}$ and, given $\Gamma_{\alpha,i}$, we define $\Gamma_{\alpha,i+1}$ by 
$$
\Gamma_{\alpha,i+1}(\overline{x}_{\alpha,i+1}) = \Gamma_{\alpha,i}(\overline{x}_{\alpha,i}) \cup \{(\exists y)\varphi(y;\overline{x}) \to \varphi(x_{\lambda};\overline{x}) : \lambda = \varphi(y;\overline{x}) \in \Lambda_{i}\}.
$$
Then $\Gamma_{\alpha}(\overline{x}_{\alpha}) = \bigcup_{i < \omega} \Gamma_{\alpha,i}(\overline{x}_{\alpha,i})$.  Note that because $\mathbb{M} \models (\exists y)[y = c]$ for each $c \in a_{<\alpha}M$,  any realization of $\Gamma_{\alpha}(\overline{x}_{\alpha})$ will contain $a_{<\alpha}M$ and will be a model by the Tarski-Vaught test.  

To complete the induction, we note that if $\alpha$ is a limit and we are given $\Gamma_{\delta}$ for all $\delta < \alpha$, then we can set $\overline{x}_{\alpha} = \bigcup_{\delta < \alpha} \overline{x}_{\delta}$ and $\Gamma_{\alpha}(\overline{x}_{\alpha}) = \bigcup_{\delta < \alpha} \Gamma_{\delta}(\overline{x}_{\delta})$, which has the desired property as the union of an elementary chain is a model.  \qed

%Let $\overline{x}=\langle \overline{x}_{\alpha}: \alpha<\lambda \rangle$ be an increasing
%continuous sequence of tuples of variables such that $\left|\overline{x}_{\alpha}\right|=\left|\alpha\right|+\left|M\right|$,
%and such that $\overline{x}_{\alpha+1}$ contains at least $\left|\alpha\right|+\left|M\right|$
%new variables not in $\overline{x}_{\alpha}$ for all $\alpha<\lambda$,
%and let $\Gamma\left(\overline{x}\right)$ be a partial type over $a_{<\lambda}M$
%saying that $a_{<\alpha}M\overline{x}_{\alpha}$ enumerates a model containing
%$Ma_{<\alpha}$ (such a type exists; one can write it using the Henkin
%property), and that $a_{\alpha}\ind_{M}^{K}\overline{a}_{<\alpha}\overline{x}_{\alpha}$
%(this can be said by symmetry since we know $\text{tp}\left(a_{\alpha}/M\right)=\text{tp}\left(a/M\right)$).

Lastly, we define $\Delta(\overline{x})$ as follows:
$$
\Delta(\overline{x}) = \Gamma(\overline{x}) \cup \{\neg \varphi(\overline{x}_{\alpha};a_{\alpha}) : \varphi(\overline{x}_{\alpha},y) \in L(M), \varphi(\overline{x}_{\alpha};a) \text{ Kim-divides over }M, \alpha < \kappa^{+}\},  
$$
where we write $\varphi(\overline{x}_{\alpha};a_{\alpha})$ to denote a formula whose variables are a finite subtuble of $\overline{x}_{\alpha}$.  To conclude, it is enough, by symmetry, to show that $\Delta\left(\overline{x}\right)$ is consistent.
By compactness, it is enough to prove this when replace $\kappa^{+}$ by a natural number $n < \omega$,
so we prove it by finding such a sequence by induction on $n$. Suppose
we found such an increasing sequence of models $N_{i}$ for $i<n$.
Let $N_{n}$ be a model containing $MN_{n-1}a_{<n}$ of size $\kappa$.
Since $a_{n}\ind_{M}a_{<n}$, we may assume by extension that $a_{n}\ind_{M}N_{n}$,
preserving all the previous types, so we are done. 
\end{proof}

\begin{cor}
\label{cor:witnessing}Suppose $T$ is NSOP$_{1}$ and $M \models T$.  If $\langle a_{i}: i<\omega \rangle$
is a Kim-Morley sequence over $M$ starting with $a_{0}=a$, then
$\varphi\left(x,a\right)$ Kim-divides over $M$ iff $\{\varphi\left(x,a_{i}\right): i<\omega\}$
is $k$-inconsistent for some $k < \omega$.
\end{cor}

\begin{proof}
One direction is Fact \ref{fact:Kim Morley is consistent}(2). The other is Theorem \ref{witnessing theorem}, since, by compactness and indiscernibility, if $\{\varphi(x;a_{i}) : i < \omega\}$ is inconsistent, it is $k$-inconsistent for some $k < \omega$.
\end{proof}

\begin{rem}
In fact, in Corollary \ref{cor:witnessing} we only need to assume $\langle a_{i} : i < \omega \rangle$ satisfies $a_{i} \ind^{K}_{M} a_{<i}$ and $a_{i} \models \text{tp}(a_{i}/M)$ for all $i < \omega$ (i.e. it is not necessary to assume that this sequence is $M$-indiscernible).  If $\varphi(x;a)$ does not Kim-divide over $M$, then $\{\varphi(x;a_{i}) : i < \omega\}$ is consistent by the independence theorem over $M$.  Conversely, if $\{\varphi(x;a_{i}) : i < \omega\}$ is not $k$-inconsistent for any $k < \omega$, then the partial type $\Gamma(y_{i} : i < \omega)$ containing, for all $i < \omega$,
\begin{itemize}
\item $y_{i} \models \text{tp}(a/M)$,
\item $\{\psi(y_{<i};y_{i}) : \psi(y_{<i};a) \in L(Ma) \text{ Kim-forks over }M\}$
\item $(\exists x)\bigwedge_{j < i} \varphi(x;y_{i})$
\end{itemize}
together with a schema asserting $\langle y_{i} : i < \omega \rangle$ is $M$-indiscernible is finitely satisfiable in $\langle a_{i} : i < \omega \rangle$ by compactness, Ramsey, and symmetry.  A realization contradicts Corollary \ref{cor:witnessing}.  
\end{rem}

\begin{cor} \label{witnesschar}
Suppose $T$ is NSOP$_{1}$, $M \models T$, and $I = \langle a_{i} : i < \omega\rangle$ is an $M$-indiscernible sequence.  The $I$ is a witness for Kim-dividing over $M$ if and only if $I$ is a $\ind^{K}$-Morley sequence over $M$.  
\end{cor}

\begin{proof}
Note that if $I$ is a witness for Kim-dividing over $M$, then $a_{i} \ind^{K}_{M} a_{<i}$ for all $i < \omega$ by symmetry:  if $\varphi(x;a_{i}) \in \text{tp}(a_{<i}/Ma_{i})$, then, by $M$-indiscernibility, $a_{<i} \models \{\varphi(x;a_{j}): j \geq i\}$ so $\varphi(x;a_{i})$ does not Kim-divide over $M$, hence $a_{<i} \ind^{K}_{M} a_{i}$.  This shows that witnesses for Kim-dividing over $M$ are $\ind^{K}$-Morley over $M$.   The other direction is Theorem \ref{witnessing theorem}.
\end{proof}

\begin{cor} \label{tms = tree}
Suppose $T$ is NSOP$_{1}$ and $M \models T$.  A sequence $I$ over $M$ is tree Morley over $M$ if and only if $I$ is a total $\ind^{K}$-Morley sequence over $M$.  
\end{cor}

\begin{proof}
By Fact \ref{sequence implications}(1), if $I$ is tree Morley over $M$ then $I$ is a total Morley sequence over $M$.  For the other direction, suppose $I = \langle a_{i} : i < \omega \rangle$ is a total $\ind^{K}$-Morley sequence over $M$ and we will show it is tree Morley over $M$.  By Fact \ref{witnessfacts}, it suffices to show $I$ is a strong witness to Kim-dividing over $M$.  Because $a_{>i} \ind^{K}_{M} a_{\leq i}$ for all $i < \omega$, if $1 \leq n < \omega$ we know $\langle (a_{n \cdot i}, a_{n \cdot i + 1}, \ldots, a_{n \cdot i + (n-1)} \rangle$ satisfies 
$$
(a_{n \cdot i}, a_{n \cdot i + 1}, \ldots, a_{n \cdot i + (n-1)})  \ind^{K}_{M} (a_{n \cdot j}, a_{n \cdot j + 1}, \ldots, a_{n \cdot j + (n-1)})_{j < i},
$$
for all $i < \omega$, or, in other words, $\langle (a_{n \cdot i}, a_{n \cdot i + 1}, \ldots, a_{n \cdot i + (n-1)} ) : i < \omega \rangle$ is an $\ind^{K}$-Morley sequence over $M$, hence a witness to Kim-dividing over $M$ by Theorem \ref{witnessing theorem}.  It follows that $I$ is a strong witness to Kim-dividing, so $I$ is tree Morley over $M$.  
\end{proof}

\section{Applications} \label{applications}

\subsection{Lifting lemmas}

The first application of the transitivity and witnessing theorems will be two `lifting lemmas' that concern $\ind^{K}$-Morley and tree Morley sequences over two bases simultaneously.  In Lemma \ref{kms over M tms over N}, we showed that if $M \prec N$ and $a \ind^{K}_{M} N$, then it is possible to construct an $\ind^{K}$-Morley sequence over $M$ beginning with $a$ which is also a tree Morley sequence over $N$.  Later, we showed under the same hypotheses in Proposition \ref{goodseq}, that we can construct a tree Morley sequence over $M$ starting with $a$ which is also an $\ind^{K}$-Morley sequence over $N$.  These raise two natural questions:  first, is it possible, under these hypotheses, to construct sequences that are tree-Morley over both bases simultaneously?  And if so, are such sequences somehow special?  We show that the answer to the first question is yes, and, moreover, address the second by showing that every $\ind^{K}$-Morley sequence (tree-Morley sequence) over $M$ beginning with $a$ is conjugate over $Ma$ to a sequence that is $\ind^{K}$-Morley (tree Morley) over $N$.  

\begin{defn}
We say that $(a_{\eta})_{\eta \in \mathcal{T}_{\alpha}}$ is $\ind^{K}$\emph{-spread out} over $M$ if for all $\eta \in \mathcal{T}_{\alpha}$ with $\text{dom}(\eta) =[\beta+1,\alpha)$ for some $\beta < \alpha$, the sequence $(a_{\unrhd \eta \frown \langle i \rangle})_{i < \omega}$ is an $\ind^{K}$-Morley sequence over $M$. 
\end{defn}

\begin{lem} \label{indkERarg}
Suppose $(a_{\eta})_{\eta \in \mathcal{T}_{\kappa}}$ is a tree of tuples, $\ind^{K}$-spread out and $s$-indiscernible over $M$.  If $\kappa$ is sufficiently large, then there is a tree $(b_{\eta})_{\eta \in \mathcal{T}_{\omega}}$, $s$-indiscernible and $\ind^{K}$-spread out over $M$, such that:
\begin{enumerate}
\item For all $w \in [\omega]^{<\omega}$, 
$$
(a_{\eta})_{\eta \in \mathcal{T}_{\kappa} \upharpoonright v} \equiv_{M} (b_{\eta})_{\eta \in \mathcal{T}_{\omega} \upharpoonright w}.  
$$
for some $v \in [\kappa \setminus \text{lim}(\kappa)]^{<\omega}$.
\item For all $w,v \in [\omega]^{<\omega}$ with $|w| = |v|$,
$$
(b_{\eta})_{\eta \in \mathcal{T}_{\omega} \upharpoonright w} \equiv_{M} (b_{\eta})_{\eta \in \mathcal{T}_{\omega} \upharpoonright v}.
$$
\end{enumerate}
\end{lem}

\begin{proof}
The proof of \cite[Lemma 5.10]{kaplan2017kim} (Fact \ref{modeling}(2)) shows that there is $(b_{\eta})_{\eta \in \mathcal{T}_{\omega}}$ satisfying (1) and (2).  As $(a_{\eta})_{\eta \in \mathcal{T}_{\kappa}}$ is $s$-indiscernible and $\ind^{K}$-spread out over $M$, (1) implies that $(b_{\eta})_{\eta \in \mathcal{T}_{\omega}}$ is $s$-indiscernible and $\ind^{K}$-spread out over $M$ as well.  See the proof of \cite[Lemma 5.10]{kaplan2017kim} (Fact \ref{modeling}(2)) for more details.  
\end{proof}

\begin{lem} \label{weak tree}
Suppose $M$ is a model and $(a_{\eta})_{\eta \in \mathcal{T}_{\alpha}}$ is a tree which is $\ind^{K}$-spread out and $s$-indiscernible over $M$ and for all $w,v \in [\alpha \setminus \text{lim}(\alpha)]^{<\omega}$ with $|w| = |v|$,
$$
(a_{\eta})_{\eta \in \mathcal{T}_{\alpha} \upharpoonright w} \equiv_{M} (a_{\eta})_{\eta \in \mathcal{T}_{\alpha} \upharpoonright v}
$$
then $(a_{\zeta_{\beta}})_{\beta \in \alpha \setminus \text{lim}(\alpha)}$ is a tree Morley sequence over $M$.
\end{lem}

\begin{proof}
The condition that for all $w,v \in [\alpha \setminus \text{lim}(\alpha)]^{<\omega}$ with $|w| = |v|$,
$$
(a_{\eta})_{\eta \in \mathcal{T}_{\alpha} \upharpoonright w} \equiv_{M} (a_{\eta})_{\eta \in \mathcal{T}_{\alpha} \upharpoonright v}
$$
implies that $(a_{\zeta_{\beta}})_{\beta \in \alpha \setminus \text{lim}(\alpha)}$ is an $M$-indiscernible sequence.  By Corollary \ref{tms = tree}, it suffices to show that $(a_{\zeta_{\beta}})_{\beta \in \alpha \setminus \text{lim}(\alpha)}$ is a total $\ind^{K}$-Morley sequence over $M$.  Fix any non-limit $\beta < \alpha$.  We know that $a_{\zeta_{\leq  \beta}}$ is a subtuple of $a_{\unrhd \zeta_{\beta}} = a_{\vartriangleright \zeta_{\beta+1\frown 0}}$ and $\langle a_{\vartriangleright \zeta_{\beta+1\frown \langle i \rangle}} : i < \omega \rangle$ is an $\ind^{K}$-Morley sequence over $M$ which is $Ma_{\zeta_{>\beta}}$-indiscernible so $a_{\zeta_{>\beta}} \ind^{K}_{M} a_{\zeta_{\leq \beta}}$ by Theorem \ref{witnessing theorem}.  
\end{proof}

\begin{prop} \label{tree lifting lemma}
Suppose $T$ is NSOP$_{1}$, $M \prec N \models T$, and $I = \langle b_{i} : i < \omega \rangle$ is a tree Morley sequence over $M$.  If $b_{0} \ind^{K}_{M} N$, then there is $I' \equiv_{Mb_{0}} I$ such that $I'$ is a tree Morley sequence over $N$.
\end{prop}

\begin{proof}
By compactness, we may stretch the sequence so that $I = \langle b_{i} : i \in \kappa \setminus \text{lim}(\kappa) \rangle$ for some cardinal $\kappa$ large relative to $|N|$.  By the chain condition, Lemma \ref{chain condition for indk}, we may also assume $I$ is $N$-indiscernible and $I \ind^{K}_{M} N$ after moving by an automorphism over $Mb_{0}$.  By induction on $\alpha \leq \kappa$, we will construct trees $(b^{\alpha}_{\eta})_{\eta \in \mathcal{T}_{\alpha}}$ and sequences $I_{\alpha} = \langle b_{\alpha,i} : i \in \kappa \setminus \text{lim}(\kappa)\rangle$ satisfying the following conditions for all $\alpha$: 
%\begin{enumerate}
%\item For all $i \in \alpha \setminus(\alpha)$, $b^{\alpha}_{\zeta_{i}} = b_{i}$ and $b^{\alpha}_{\emptyset} = b_{\alpha}$ for $\alpha$ successor.
%\item $(b^{\alpha}_{\eta})_{\eta \in \mathcal{T}_{\alpha}}$ is $\ind^{K}$-spread out over $N$ and $s$-indiscernible over $N(b_{i})_{i \geq \alpha}$.  
%\item $I_{>\alpha}$ is $M(b^{\alpha}_{\eta})_{\eta \in \mathcal{T}_{\alpha}}$-indiscernible and $I_{\geq \alpha}$ is $M(b^{\alpha}_{\eta})_{\eta \in \mathcal{T}_{\alpha}}$-indiscernible for $\alpha$ limit.  
%\item If $\alpha < \beta$, then $b^{\alpha}_{\eta} = b^{\beta}_{\iota_{\alpha \beta}(\eta)}$ for $\eta \in \mathcal{T}_{\alpha}$.  
%\item $(b^{\alpha}_{\eta})_{\eta \in \mathcal{T}_{\alpha}} \ind^{K}_{M} N$.  
%\end{enumerate}
\begin{enumerate}
\item For all non-limit $i \leq \alpha$, $b^{\alpha}_{\zeta_{i}} = b_{\alpha,i}$.
\item $(b^{\alpha}_{\eta})_{\eta \in \mathcal{T}_{\alpha}}$ is $\ind^{K}$-spread out over $N$ and $s$-indiscernible over $NI_{\alpha,>\alpha}$.  
\item If $\beta < \alpha$, $I_{\alpha} \equiv_{M} I_{\beta}$ and $I_{0} = I$.  
\item $I_{\alpha,>\alpha}$ is $M(b^{\alpha}_{\eta})_{\eta \in \mathcal{T}_{\alpha}}$-indiscernible.
\item If $\alpha < \beta$, then $b^{\alpha}_{\eta} = b^{\beta}_{\iota_{\alpha \beta}(\eta)}$ for $\eta \in \mathcal{T}_{\alpha}$.  
\item $I_{\alpha}(b^{\alpha}_{\eta})_{\eta \in \mathcal{T}_{\alpha}} \ind^{K}_{M} N$.  
\item $b^{\alpha}_{\eta} \equiv_{N} b_{\alpha,i} \equiv_{N} b_{0}$ for all $\eta \in \mathcal{T}_{\alpha}$, $i \in \kappa \setminus \text{lim}(\kappa)$.  
\end{enumerate}
For the base case, we define $b^{0}_{\emptyset} = b_{0}$ and $I_{0} = I$, which satisfies all the demands.  Next, suppose we are given $(b^{\beta}_{\eta})_{\eta \in \mathcal{T}_{\beta}}$ for all $\beta \leq \alpha$ and we will construct $(b^{\alpha+1}_{\eta})_{\eta \in \mathcal{T}_{\alpha+1}}$.  By (6) and Lemma \ref{goodseq}, we may obtain a sequence $J = \langle (b^{\alpha}_{\eta,i})_{\eta \in \mathcal{T}_{\alpha}} : i < \omega \rangle$ with $(b^{\alpha}_{\eta,0})_{\eta \in \mathcal{T}_{\alpha}} = (b^{\alpha}_{\eta})_{\eta \in \mathcal{T}_{\alpha}}$ which is tree-Morley over $M$ and $\ind^{K}$-Morley over $N$.  As $J$ is a tree Morley sequence over $M$ which is $N$-indiscernible, we have:
\begin{equation} \tag{a}
J \ind^{K}_{M} N,
\end{equation} 
Likewise, as $I_{\alpha, > \alpha}$ is a tree Morley sequence over $M$ by (3) which is $M(b^{\alpha}_{\eta})_{\eta \in \mathcal{T}_{\alpha}}$-indiscernible by (4), we have $I_{\alpha,>\alpha} \ind^{K}_{M} (b^{\alpha}_{\eta})_{\eta \in \mathcal{T}_{\alpha}}$.  By the chain condition (Lemma \ref{chain condition for indk}), there is $I'_{\alpha,>\alpha} \equiv_{M(b^{\alpha}_{\eta})_{\eta \in \mathcal{T}_{\alpha}}} I_{\alpha, > \alpha}$ so that $J$ is $MI'_{\alpha,>\alpha}$-indiscernible and also:
\begin{equation}\tag{b}
I'_{\alpha,>\alpha} \ind^{K}_{M} J.
\end{equation}  
Choose $N'$ so that $NI_{\alpha,>\alpha} \equiv_{M} N'I'_{\alpha,>\alpha}$.  By (6) and invariance, we have
\begin{equation} \tag{c}
N' \ind^{K}_{M} I'_{\alpha,>\alpha}.
\end{equation} 
By (a), (b), and (c), we may apply the independence theorem to find a model $N''$ with $N'' \equiv_{MJ} N$, $N'' \equiv_{MI'_{\alpha,>\alpha}} N'$, and $N'' \ind^{K}_{M} I'_{\alpha,>\alpha}J$.  Now choose $I''_{\alpha,>\alpha} = \langle b''_{\alpha,i} : i \in \kappa \setminus (\text{lim}(\kappa) \cup \alpha) \rangle$ so that $N I''_{\alpha,>\alpha}\equiv_{MJ} N'' I'_{\alpha,>\alpha}$.  

Define a tree $(c_{\eta})_{\eta \in \mathcal{T}_{\alpha+1}}$ by setting $c_{\emptyset} = b''_{\alpha, \alpha+1}$ and $c_{\langle i \rangle \frown \eta} = b^{\alpha}_{\eta,i}$ for all $\eta \in \mathcal{T}_{\alpha}$ and $i < \omega$.  With this definition, we have $N \ind^{K}_{M} (c_{\eta})_{\eta \in \mathcal{T}_{\alpha+1}} I''_{\alpha,>\alpha+1}$.  Let $(c'_{\eta})_{\eta \in \mathcal{T}_{\alpha+1}}$ be a tree which is $s$-indiscernible over $NI''_{\alpha,> \alpha+1}$ locally based on $(c_{\eta})_{\eta \in \mathcal{T}_{\alpha}}$.  By symmetry and finite character, we have $N \ind^{K}_{M} (c'_{\eta})_{\eta \in \mathcal{T}_{\alpha+1}}I''_{\alpha,>\alpha+1}$.  Finally, let $I'''_{\alpha,>\alpha+1} = \langle b'''_{\alpha,i} : i \in \kappa \setminus (\text{lim}(\kappa) \cup (\alpha+1)) \rangle$ be an $N(c'_{\eta})_{\eta \in \mathcal{T}_{\alpha+1}}$-indiscernible sequence locally based on $I''_{\alpha,>\alpha+1}$.  By symmetry, we have $N \ind^{K}_{M} I'''_{\alpha,>\alpha}(c'_{\eta})_{\eta \in \mathcal{T}_{\alpha+1}}$.  Note that, by (2) and the construction, $(c'_{\eta})_{\eta \in \mathcal{T}_{\alpha+1}}$ is $\ind^{K}$-spread out over $N$ and $s$-indiscernible over $NI'''_{\alpha,>\alpha+1}$.  Moreover, by (2) and the construction, there is an automorphism $\sigma \in \text{Aut}(\mathbb{M}/N)$ such that $\sigma(c'_{\langle 0 \rangle \frown \eta}) = b^{\alpha}_{\eta}$ for all $\eta \in \mathcal{T}_{\alpha}$ so we will define $b^{\alpha+1}_{\eta} = \sigma(c'_{\eta})$ for all $\eta \in \mathcal{T}_{\alpha+1}$.  Likewise, we define $I_{\alpha+1} = \langle b_{\alpha+1,i} : i \in \kappa \setminus \text{lim}(\kappa)\rangle$ by $b_{\alpha+1,i} = b^{\alpha+1}_{\zeta_{i}}$ for non-limit $i \leq \alpha+1$ and $b_{\alpha+1,i} = \sigma(b'''_{\alpha,i})$ for non-limit $i > \alpha+1$.  It is immediate that this construction satisfies (6) and (7) by induction and the construction of $N''$.  To check (3), note that, by induction, using (1),(2), and (3), for any function $\eta : \alpha \to \omega$, we have $(b^{\alpha}_{\eta |_{[\beta,\alpha)}})_{\beta \in \alpha \setminus \text{lim}(\alpha)} I_{\alpha, > \alpha} \equiv_{M} I$, and therefore, for any $i < \omega$, we have 
$$
(b^{\alpha}_{\eta|_{[\beta,\alpha)},i})_{\beta \in \alpha \setminus \text{lim}(\alpha)} I''_{\alpha,>\alpha} \equiv_{M} (b^{\alpha}_{\eta|_{[\beta,\alpha)},i})_{\beta \in \alpha \setminus \text{lim}(\alpha)} I'_{\alpha,>\alpha} \equiv_{M} I.
$$
By the definition of $(c_{\eta})_{\eta \in T_{\alpha+1}}$ and $s$-indiscernibility over $M$, it follows that, for any function $\eta' : (\alpha+1) \to \omega$,
$$
(c'_{\eta'|_{[\beta,\alpha+1)}})_{\beta \in (\alpha+1) \setminus \text{lim}(\alpha+1)}I'''_{\alpha,>\alpha+1} \equiv_{M} (c_{\eta'|_{[\beta,\alpha+1)}})_{\beta \in (\alpha+1) \setminus \text{lim}(\alpha+1)}I''_{\alpha,>\alpha+1} \equiv_{M} I,
$$
from which (3) follows.  
The remaining constraints are easily seen to be satisfied by the construction.

Now for $\delta$ limit, if we are given $(b^{\alpha}_{\eta})_{\eta \in \mathcal{T}_{\alpha}}$ for $\alpha < \delta$, we may define $b^{\delta}_{\iota_{\alpha \delta}(\eta)} = b^{\alpha}_{\eta}$ for all $\alpha < \delta$ and $\eta \in \mathcal{T}_{\alpha}$.    We define $I_{\delta}$ as follows:  $I_{\delta,<\delta}$ will be defined by $b_{\delta,i} = b_{i,i}$ for all non-limit $i < \delta$.  By (1),(3), and induction, we have $I_{\delta,<\delta} \equiv_{M} I_{<\delta}$.  Choose $J$ so that $I_{\delta,<\delta}J \equiv_{M} I_{<\delta}I_{>\delta}$.  Write $\overline{x}$ for $\langle x_{i} : i \in \kappa \setminus (\text{lim}(\kappa) \cup \delta) \rangle$ and $\varphi(\overline{x};c,n)$ to denote any formula where the variables are a finite subtuple of $\overline{x}$.  By (6), induction, and compactness, the partial type, which contains $\text{tp}_{\overline{x}}(J/MI_{\delta,<\delta})$ and $\{\neg \varphi(\overline{x},c;n) : c \in (b^{\delta}_{\eta})_{\eta \in \mathcal{T}_{\delta}}, n \in N, \varphi(x,y;n) \text{ Kim-divides over }M\}$, and naturally expresses that both $(b^{\delta}_{\eta})_{\eta \in \mathcal{T}_{\delta}}$ is $s$-indiscernible over $N\overline{x}$ and $\overline{x}$ is $N(b^{\delta}_{\eta})_{\eta \in \mathcal{T}_{\alpha}}$-indiscernible, is consistent.  Let $I_{\delta,>\delta}$ is a realization of this type, completing the definition of $I_{\delta}$. It is easy to check that these are well-defined and satisfy all of the requirements by induction and the finite character of Kim-independence.  

This completes the recursion and yields $(b^{\kappa}_{\eta})_{\eta \in \mathcal{T}_{\kappa}}$ likewise defined by $b^{\kappa}_{\iota_{\alpha \kappa}(\eta)} = b^{\alpha}_{\eta}$ for all $\alpha < \kappa$ and $\eta \in \mathcal{T}_{\alpha}$.  Apply Lemma \ref{indkERarg} to obtain a tree $(c_{\eta})_{\eta \in \mathcal{T}_{\omega}}$ so that for all $w \in [\omega]^{<\omega}$, there is $v \in [\kappa \setminus \text{lim}(\kappa)]^{<\omega}$ such that
$$
(b^{\kappa}_{\eta})_{\eta \in \mathcal{T}_{\kappa} \upharpoonright v} \equiv_{N} (c_{\eta})_{\eta \in \mathcal{T}_{\omega} \upharpoonright w},
$$
and, moreover, for all $w,v \in [\omega]^{<\omega}$ with $|w| = |v|$,
$$
(c_{\eta})_{\eta \in \mathcal{T}_{\omega} \upharpoonright w} \equiv_{N} (c_{\eta})_{\eta \in \mathcal{T}_{\omega} \upharpoonright v}.
$$
By an automorphism, we can assume $c_{\zeta_{0}} = b_{0}$, hence, setting $I' = \langle c_{\zeta_{i}} : i < \omega \rangle$, we have $I' \equiv_{Mb_{0}} I$.  Moreover, by Lemma \ref{weak tree}, $I'$ is a tree Morley sequence over $N$, completing the proof.    
\end{proof}

%\begin{rem}
%In the previous proof, we only use that $I$ is a strong witness to Kim-dividing over $M$ and so this proposition generalizes \cite[Proposition 7.9]{kaplan2017kim}, which proves that a strong witness to Kim-dividing over $M$ must be a tree Morley sequence over $M$, which follows, using that $b \ind^{K}_{M} M$ always.  An alternative proof is possible, taking advantage of the associated Morley tree, by appealing to Proposition \ref{lifting lemma}.  This route is somewhat shorter but less elementary so we have opted for the proof given above.  
%\end{rem}

The second lifting lemma, is an analogue of Proposition \ref{tree lifting lemma} for $\ind^{K}$-Morley sequences.  

\begin{prop}\label{lifting lemma}
Suppose $T$ is NSOP$_{1}$, $M \prec N \models T$, and $I = \langle b_{i} : i < \omega \rangle$ is an $\ind^{K}$-Morley sequence over $M$.  If $b_{0} \ind^{K}_{M} N$, then there is $I' \equiv_{Mb_{0}} I$ satisfying the following conditions:
\begin{enumerate}
\item $I' \ind^{K}_{M} N$
\item $I'$ is an $\ind^{K}$-Morley sequence over $N$.  
\end{enumerate}
\end{prop}

\begin{proof}
For $i < \omega$, let $q_{i}(x_{j} : j \leq i) = \text{tp}(b_{\leq i}/M)$ and let $p(x;N) = \text{tp}(b_{0}/N)$.  For a natural number $K$, define $\Gamma_{K}$ to be the partial type defined as the union of the following: 
\begin{enumerate}[(a)]
\item $q_{K}(x_{i} : i \leq K)$.
\item $\bigcup_{i \leq K} p(x_{i};N)$.
%\item  $\{\neg \varphi(x_{<i},c;x_{i}) : i \leq \alpha, \varphi(x_{<i},y;x_{i}) \in L(M), c \in N, \varphi(x_{<i},y;b) \text{Kim-divides over }M\}$.
\item $\{\neg \varphi(x_{\leq K};c) : \varphi(x_{\leq K},y) \in L(M),c \in N,\varphi(x_{\leq K};c) \text{ Kim-divides over }M\}$.
\item $\{\neg \varphi(x_{<i};x_{i}) : i \leq K,\varphi(x_{<i};x_{i}) \in L(N), \varphi(x_{<i};b_{0}) \text{ Kim-divides over }N\}$.
\end{enumerate}
By Ramsey and compactness, it is enough to show the consistency of $\Gamma = \bigcup_{K < \omega} \Gamma_{K}$. 

As $b_{0} \ind^{K}_{M} N$, $\Gamma_{0}$ is consistent.  Suppose $\Gamma_{K}$ is consistent and we will that show $\Gamma_{K+1}$ is consistent.  Let $\Delta(x_{0},\ldots, x_{K+1}) \subseteq \Gamma_{K+1}$ be the partial type defined as the union of the following:
\begin{enumerate}
\item $q_{K+1}(x_{i} : i \leq K+1)$.
\item $\bigcup_{i \leq K+1} p(x_{i};N)$.
\item $\{\neg \varphi(x_{\leq K+1};c) : \varphi(x_{\leq K+1},y) \in L(M),c \in N,\varphi(x_{\leq K+1};c) \text{ Kim-divides over }M\} $.
\item $\{\neg \varphi(x_{<i};x_{i}) : i < K,\varphi(x_{<i};x_{i}) \in L(N), \varphi(x_{<i};b_{0}) \text{ Kim-divides over }N\} $.
%\item $\{\neg \varphi(x_{<i},c;x_{i}) : i \leq K+1, \varphi(x_{<i},y;x_{i}) \in L(M), c \in N, \varphi(x_{<i},y;b) \text{Kim-divides over }M\} $.
\end{enumerate}
Note that $\Delta$ is identical to $\Gamma_{K+1}$ except that in the final set of formulas, $i$ is taken to be less than $K$ rather than $K+1$.  

\textbf{Claim 1}:  $\Delta$ is consistent.  

\emph{Proof of claim}:  Let $(b'_{0},\ldots, b'_{K}) \models \Gamma_{K}$ and choose $b'_{K+1}$ so that $b'_{\leq K+1} \equiv_{M} b_{\leq K+1}$.  Next, choose a model $N'$ so that $b'_{K+1}N' \equiv_{M} b_{0}N$.  Now by definition of $\Gamma_{K}$ and symmetry, we have $N \ind^{K}_{M} b'_{\leq K}$ and our assumption that $b_{0} \ind^{K}_{M} N$ implies $N' \ind^{K}_{M} b'_{K+1}$ by symmetry and invariance.  Moreover, because $I$ is an $\ind^{K}$-Morley sequence, we likewise have $b'_{K+1} \ind^{K}_{M} b'_{\leq K}$.  Therefore, we may apply the independence theorem to find $N'' \models \text{tp}(N/Mb'_{\leq K}) \cup \text{tp}(N'/Mb'_{K+1})$ such that $N'' \ind^{K}_{M} b'_{\leq K+1}$.  There is an automorphism $\sigma \in \text{Aut}(\mathbb{M}/Mb'_{\leq K})$ with $\sigma(N'') = N$.  Let $b''_{K+1} = \sigma(b'_{K+1})$.  Then $(b'_{0},\ldots, b'_{K},b''_{K+1}) \models \Delta$.\qed

\textbf{Claim 2}:  Suppose $J = \langle c_{K+1,i} : i < \omega \rangle$ is an $\ind^{K}$-Morley sequence over $M$ with $b_{0} = c_{K+1,0}$.  If $J$ is $N$-indiscernible and $J \ind^{K}_{M} N$, then $\bigcup_{i < \omega} \Delta(x_{0},\ldots, x_{K},c_{K+1,i})$ is consistent.  

\emph{Proof of claim}:  Choose $(c_{0},\ldots, c_{K})$ so that $(c_{0},\ldots, c_{K},c_{K+1,0}) \models \Delta$.  Then $c_{K+1,0} \ind^{K}_{M} c_{\leq K}$ so there is $J' \equiv_{Mc_{K+1,0}} J$ such that $J'$ is $Mc_{\leq K}$-indiscernible and $J' \ind^{K}_{M} c_{\leq K}$, by the chain condition for $\ind^{K}$-Morley sequences Lemma \ref{chain condition for indk}.  Moreover, by definition of $\Delta$, we have $c_{\leq K} \ind^{K}_{M} N$.  By assumption, $J \ind^{K}_{M} N$ and, since $J \equiv_{M} J'$, we may, therefore,  apply the strengthened independence theorem, Fact \ref{fact:Kim Morley is consistent}(3), to find $J_{*}$ that simultaneously realizes $\text{tp}(J'/Mc_{\leq K})$, to satisfy condition $(1)$ in the definition of $\Delta$, and $\text{tp}(J/N)$, to satisfy condition (2), and, moreover, such that $N \ind^{K}_{M} J_{*}c_{\leq K}$, to satisfy (3).   Choose $c'_{\leq K}$ so that $c'_{\leq K}J \equiv_{N} c_{\leq K}J_{*}$.  Then, by definition of $\Delta$, $c'_{\leq K} \models \bigcup_{i < \omega} \Delta(x_{0},\ldots, x_{K},c_{K+1,i})$.\qed

To conclude, we use Proposition \ref{tree lifting lemma} to select $J = \langle b'_{K+1,i} : i < \omega \rangle$ which is simultaneously a tree Morley sequence over $M$ and a tree Morley sequence over $N$ with $J \ind^{K}_{M} N$.  In particular, $J$ is an $\ind^{K}$-Morley sequence over $M$.  Hence, by Claim 2, 
$$
\bigcup_{i < \omega} \Delta(x_{0},\ldots, x_{K}, b'_{K+1,i})
$$
is consistent, so we may realize it with $(b'_{0},\ldots, b'_{K})$.  By compactness and Ramsey, we may additionally assume that $\langle b_{K+1,i} : i < \omega \rangle$ is $Nb'_{\leq K}$-indiscernible.  Put $b'_{K+1} = b'_{K+1,0}$.  It follows, by Kim's Lemma for tree Morley sequences (Fact \ref{witnessfacts}(1)), that $b'_{K+1} \ind^{K}_{N} b'_{\leq K}$ and, therefore by definition of $\Delta$, $b'_{<i} \ind^{K}_{N} b'_{i}$ for all $i \leq K+1$.  Additionally, by definition of $\Delta$, we have $b'_{\leq K+1} \models q_{K+1}$, $b'_{i} \models p(x;N)$ for all $i \leq K+1$, and $b'_{\leq K+1} \ind^{K}_{M} N$.  This shows $b'_{\leq K+1} \models \Gamma_{K+1}$.  
\end{proof}

\subsection{Doubly local character}

In \cite[Lemma 3.7]{24-Kaplan2017}, the following variant of local character was established:  if $\langle M_{i} : i < \alpha \rangle$ is an increasing sequence of elementary submodels of $N$ and $p \in S(N)$ does not Kim-divide over $M_{i}$ for all $i < \alpha$, then $p$ does not Kim-divide over $M_{\alpha}$.  The proof there uses the fact that $p$ is a complete type in an essential way, which left open whether or not a local version of this form of local character (hence the name \emph{doubly} local character) might also hold, where the type $p$ is replaced by a formula over $N$.  We prove this in Proposition \ref{doubly local character}, answering \cite[Question 3.17]{24-Kaplan2017} .  

\begin{defn}
Suppose $\alpha$ is an ordinal and $\mathcal{U}$ is an ultrafilter on $\alpha$.  Given a sequence of sequences $\langle \overline{b}_{i} : i < \omega \rangle$, where $\overline{b}_{i} = \langle b_{i,j} : j < \omega \rangle$ for all $i < \alpha$, we say that $\overline{a}$ is a $\mathcal{U}$\emph{-average} of $\langle \overline{b}_{i} : i < \alpha \rangle$ over $A$ if, for all $n< \omega$ and $\varphi(x_{0},\ldots, x_{n-1}) \in L(A)$, we have 
$$
\mathbb{M} \models \varphi(a_{<n}) \iff \{ i  \in \alpha : \mathbb{M} \models \varphi(b_{i,<n})\} \in \mathcal{U}.
$$
\end{defn}

It is an easy exercise to show that $\mathcal{U}$-averages exist for any sequence of sequences and parameter sets $A$.  

\begin{lem}  \label{lemma:limit of heirs}
Suppose we are given:
\begin{enumerate}
\item An increasing continuous elementary chain $\langle M_{i} : i \leq \alpha \rangle$ of models of $T$.
\item For every $i< \alpha$, $\bar{b}_{i}=\langle b_{i,j}: j<\omega\rangle$ is an indiscernible heir sequence over $M_{i}$.
\item For all $i \leq j$, $b_{i,0} \equiv_{M_{i}} b_{j,0}$.  
\end{enumerate}
Then for any ultrafilter $\mathcal{U}$ on $\alpha$ concentrating on end segments of $\alpha$, if $\bar{a}=\langle a_{j}: j<\omega \rangle$ realizes the $\mathcal{U}$-average of $\langle \bar{b}_{i} : i<\alpha\rangle$
over $M_{\alpha}$, then $\langle a_{j} : j<\omega \rangle$ is an heir
sequence over $M_{\alpha}$ such that $a_{0}\equiv_{M_{i}} b_{i,0}$ for all $i<\alpha$.
\end{lem}

\begin{proof}
The fact that $\bar{a}$ is an indiscernible sequence over $M_{\alpha}$ and $a_{0} \equiv_{M_{i}} b_{i,0}$ is
clear by construction. We are left with showing that $\bar{a}$ is
an heir sequence over $M_{\alpha}$. Suppose that $\psi\left(a_{j},a_{<j},m\right)$
where $m\in M_{\alpha}$ and $\psi\left(y,z,w\right)$ is an $L$-formula. Then for some $i<\alpha$ such that $m\in M_{i}$, $\psi\left(b_{i,j},b_{i,<j},m\right)$
holds. Hence for some $n\in M_{i}$, $\psi\left(b_{i,j},n,m\right)$
holds. Hence $\psi\left(a_{0},n,m\right)$ holds (as $b_{i,j}\equiv_{M_{i}}a_{0}$)
and hence $\psi\left(a_{j},n,m\right)$ holds. 
\end{proof}

\begin{defn}
Suppose $M$ is a model and $k<\omega$. Say that a formula $\varphi\left(x,a\right)$
\emph{$k$-Kim-divides over $M$} if there is an $\ind^{K}$-Morley sequence
$\langle a_{i}: i<\omega \rangle$ over $M$ starting with $a_{0}=a$ such
that $\{\varphi\left(x,a_{i}\right): i<\omega\}$ is $k$-inconsistent. 
\end{defn}

\begin{rem}
There is a choice involved in defining \emph{$k$-Kim-dividing}, since it is not known if, in an NSOP$_{1}$ theory, a formula that $k$-divides with respect to some $\ind^{K}$-Morley sequence will also $k$-divide along a Morley sequence in a global invariant type.  The above definition differs from the one implicitly used in \cite{24-Kaplan2017}, but in light of Corollary \ref{witnesschar} this definition seems reasonably canonical, given that any sequence which is a witness to Kim-dividing over $M$ will be an $\ind^{K}$-Morley sequence over $M$ and hence $\varphi\left(x,a\right)$
$k$-Kim-divides over $M$ for some $k<\omega$ iff $\varphi\left(x,a\right)$
Kim-divides over $M$. 
\end{rem}

\begin{prop} \label{doubly local character}
Suppose that $\langle M_{i}: i<\alpha\rangle$ is an increasing sequence
of models of $T$ with union $M=\bigcup_{i<\alpha}M_{i}$. Let $\varphi\left(x,y\right)$
be some formula (over $\emptyset$) and $a\in\mathbb{M}^{y}$. Fix some $k<\omega$. 
\begin{enumerate}
\item If $\varphi\left(x,a\right)$ Kim-divides over $M$ then $\varphi\left(x,a\right)$
Kim-divides over $M_{i}$ for some $i<\alpha$. 
\item If $\varphi\left(x,a\right)$ $k$-Kim-divides over $M_{i}$ for
all $i<\alpha$ then $\varphi\left(x,a\right)$ $k$-Kim-divides
over $M$.
\end{enumerate}
\end{prop}

\begin{proof}
Note that this proposition, once proved, is immediately also true when
we allow parameters from $M$ inside $\varphi$, as long as we assume
these parameters are from $M_{0}$, by adding constants to the language. As the statement is trivial when $\alpha$ is a successor, we may assume $\alpha$ is a limit ordinal.  

(1) Suppose that $\varphi\left(x,a\right)$ does not Kim-divide
over any $M_{i}$. For $i<\alpha$, let $\bar{b}_{i}=\langle b_{i,j} : j<\omega \rangle$
be an indiscernible heir sequence starting with $b_{i,0}=a$ over
$M_{i}$ (such a sequence exists, by e.g., taking a coheir sequence
in reverse). In particular, $\bar{b}$ is an $\ind^{K}$-Morley sequence by
symmetry. By Corollary \ref{cor:witnessing}, $\{\varphi\left(x,b_{i,j}\right): j<\omega\}$
is consistent. Let $\mathcal{U}$ be an ultrafilter on $\alpha$, concentrating on end-segments of $\alpha$.
Let $\bar{a}=\langle a_{j} : j<\omega \rangle$ be a $\mathcal{U}$-average
of $\langle \bar{b}_{i} : i<\alpha \rangle$ over $M$.  Then Lemma \ref{lemma:limit of heirs} 
and symmetry imply that $\bar{a}$ is a $\ind^{K}$-Morley sequence over
$M$, and by construction $\{\varphi\left(x,a_{j}\right): j<\omega\}$
is consistent. By Corollary \ref{cor:witnessing}, $\varphi\left(x,a\right)$
does not Kim-divide over $M$.

(2)  Suppose that $\varphi\left(x,a\right)$ $k$-Kim-divides over $M_{i}$
for all $i<\alpha$. For $i<\alpha$ let $\bar{b}_{i}=\langle b_{i,j} : j<\omega \rangle$
be a $\ind^{K}$-Morley sequence over $M_{i}$ witnessing this, i.e., $\{\varphi\left(x,b_{i,j}\right): j<\omega\}$
is $k$-inconsistent and $b_{i,0}=a$. As above, we let $\mathcal{U}$ be an
ultrafilter on $\alpha$, concentrating on end-segments, and let $\bar{a}= \langle a_{j} : j<\omega\rangle$
be a $\mathcal{U}$-average of $\langle \bar{b}_{i} : i<\alpha \rangle$ over
$M$. Then $\overline{a}$ is an $M$-indiscernible sequence in $\text{tp}(a/M)$ such that $\{\varphi(x;a_{i}) : i < \omega\}$ is $k$-inconsistent, so is it is enough to show that $\bar{a}$ is a $\ind^{K}$-Morley sequence
over $M$. By symmetry it is enough to show that $a_{<j}\ind_{M}^{K}a_{j}$
for all $j<\omega$. Suppose this is not the case, i.e., $\models \psi\left(a_{<j},a_{j},m\right)$ for some $m\in M$ and $j < \omega$, where $\psi\left(z,y,w\right)$ is an $L$-formula
and $\psi\left(z,a_{j},m\right)$ Kim-divides over $M$, so also $\psi\left(z,a,m\right)$
Kim-divides over $M$. Hence, for some $S\in\mathcal{U}$, $m\in M_{i}$
and $\models \psi\left(b_{i,<j},b_{i,j},m\right)$ for all $i\in S$.
Let $\langle N_{i}: i<\beta \rangle$ be an increasing enumeration of $\langle M_{i} : i\in S \rangle$.
By (1), applied to $\langle N_{i} : i<\beta \rangle$ and the formula
$\psi\left(z,a,m\right)$, we have that $\psi(z,a,m)$ Kim-divides over $M_{i}$ for some $i\in S$.
Hence also $\psi\left(z,b_{i,j},m\right)$ Kim-divides over $M_{i}$
(as $b_{i,j}\equiv_{M_{i}}a$), contradicting the fact that $\bar{b}_{i}$
is a $\ind^{K}$-Morley sequence over $M_{i}$. 
\end{proof}

%\begin{rem}
%EXPLAIN HOW ANSWERS QUESTION
%\end{rem}
%
\subsection{Reformulating the Kim-Pillay-style characterization}

Our final application will be an easy corollary of witnessing for $\ind^{K}$-Morley sequences, allowing us to give a more satisfying formulation of the Kim-Pillay-style characterization of Kim-independence.  In \cite[Proposition 5.3]{ArtemNick}, a Kim-Pillay-style criterion was given for NSOP$_{1}$, consisting of 5 axioms for an abstract independence relation on subsets of the monster model.  Later, it was shown in \cite[Theorem 9.1]{24-Kaplan2017} that any independence relation $\ind$ satisfying these axioms must \emph{strengthen} $\ind^{K}$ in the sense that whenever $M \models T$ and $a \ind_{M} b$, then also $a \ind^{K}_{M} b$.  In order to characterize $\ind^{K}$, it was necessary to add an additional axiom to the list called \emph{witnessing}:  if $a \nind_{M} b$ witnessed by $\varphi(x;b)$ and $(b_{i})_{i < \omega}$ is a Morley sequence over $M$ in a global $M$-invariant (or even $M$-finitely satisfiable) type extending $\text{tp}(b/M)$, then $\{\varphi(x;b_{i}) : i < \omega\}$ is inconsistent.  Though useful in practice, this is somewhat unsatisfying, as it requires reference to independence notions like invariance or finite satisfiability instead of a property intrinsic to $\ind$.  

\begin{thm} \label{criterion}
Assume there is an \(\text{Aut}(\mathbb{M})\)-invariant ternary relation \(\ind\) on small subsets of the monster \(\mathbb{M} \models T\) which satisfies the following properties, for an arbitrary \(M \models T\) and arbitrary tuples from $\mathbb{M}$.
\begin{enumerate}
\item Strong finite character: if \(a \nind_{M} b\), then there is a formula \(\varphi(x,b,m) \in \text{tp}(a/bM)\) such that for any \(a' \models \varphi(x,b,m)\), \(a' \nind_{M} b\). 
\item Existence over models:  \(M \models T\) implies \(a \ind_{M} M\) for any \(a\).
\item Monotonicity: \(aa' \ind_{M} bb'\) \(\implies\) \(a \ind_{M} b\).
\item Symmetry: \(a \ind_{M} b \iff b \ind_{M} a\).
\item The independence theorem: \(a \ind_{M} b\), \(a' \ind_{M} c\), \(b \ind_{M} c\) and $a \equiv_{M} a'$ implies there is $a''$ with $a'' \equiv_{Mb} a$, $a'' \equiv_{Mc} a'$ and $a'' \ind_{M} bc$.
\item $\ind$-Morley sequences are witnesses:  if $M \models T$ and $I = (b_{i})_{i < \omega}$ is an $M$-indiscernible sequence with $b_{0} = b$ satisfying $b_{i} \ind_{M} b_{<i}$, then whenever $a \nind_{M} b$, there is $\varphi(x;m,b) \in \text{tp}(a/Mb)$ such that $\{\varphi(x;m,b_{i}) : i < \omega\}$ is inconsistent.  
\end{enumerate}
Then \(T\) is NSOP\(_{1}\) and $\ind = \ind^{K}$ over models, i.e. if $M \models T$, $a \ind_{M} b$ if and only if $a \ind^{K}_{M} b$.
\end{thm}

\begin{proof}
Because $\ind$ satisfies axioms (1) through (5), it follows that $T$ is NSOP$_{1}$ and for any $M \models T$, if $a \ind_{M} b$ then $a \ind^{K}_{M} b$, by \cite[Theorem 9.1]{24-Kaplan2017}.  For the other direction, suppose $a \ind^{K}_{M} b$.  Let $I = \langle b_{i} : i < \omega \rangle$ be an $M$-finitely satisfiable Morley sequence over $M$ with $b_{0} = b$.  As $a \ind^{K}_{M} b$, we find $a' \equiv_{Mb} a$ so that $I$ is $Ma'$-indiscernible.  By \cite[Claim in proof of Proposition 5.3]{ArtemNick}, any relation $\ind$ satisfying (1)--(4), we have $c \ind^{u}_{M} d$ implies $c \ind_{M} d$.  Therefore, the sequence $I$ is, in particular, an $\ind$-Morley sequence over $M$ and $a' \models \bigcup_{i < \omega} p(x;b_{i})$ so $a \ind_{M} b$ by (6).  
\end{proof}

\begin{rem}
In any NSOP$_{1}$ theory, $\ind^{K}$ satisfies properties (1)--(6), by Fact \ref{basic kimindep facts} and Theorem \ref{witnessing theorem}, so the existence of such a relation characterizes NSOP$_{1}$ theories.   
\end{rem}

\bibliographystyle{alpha}
\bibliography{ms.bib}{}

\begin{thebibliography}{KKS14}

\bibitem[CR16]{ArtemNick}
Artem Chernikov and Nicholas Ramsey.
\newblock On model-theoretic tree properties.
\newblock {\em Journal of Mathematical Logic}, page 1650009, 2016.

\bibitem[DS04]{dvzamonja2004maximality}
Mirna D{\v{z}}amonja and Saharon Shelah.
\newblock On $\vartriangleleft^{*}$-maximality.
\newblock {\em Annals of Pure and Applied Logic}, 125(1):119--158, 2004.

\bibitem[Kim01]{kim2001simplicity}
Byunghan Kim.
\newblock Simplicity, and stability in there.
\newblock {\em The Journal of Symbolic Logic}, 66(02):822--836, 2001.

\bibitem[KKS14]{KimKimScow}
Byunghan Kim, Hyeung-Joon Kim, and Lynn Scow.
\newblock Tree indiscernibilities, revisited.
\newblock {\em Arch. Math. Logic}, 53(1-2):211--232, 2014.

\bibitem[KR18]{kruckman2018generic}
Alex Kruckman and Nicholas Ramsey.
\newblock Generic expansion and skolemization in $\text{NSOP}_1$ theories.
\newblock {\em Annals of Pure and Applied Logic}, 169(8):755--774, 2018.

\bibitem[KR20]{kaplan2017kim}
Itay Kaplan and Nicholas Ramsey.
\newblock On \text{K}im-independence.
\newblock {\em Journal of the European Mathematical Society}, 22(5):1423--1474,
  2020.

\bibitem[KRS19]{24-Kaplan2017}
Itay Kaplan, Nicholas Ramsey, and Saharon Shelah.
\newblock Local character of {K}im-independence.
\newblock {\em Proc. Amer. Math. Soc.}, 147(4):1719--1732, 2019.

\bibitem[She90]{shelah1990classification}
Saharon Shelah.
\newblock {\em Classification theory: and the number of non-isomorphic models}.
\newblock Elsevier, 1990.

\bibitem[Win75]{winkler1975model}
Peter Winkler.
\newblock Model-completeness and skolem expansions.
\newblock {\em Model Theory and Algebra}, pages 408--463, 1975.

\end{thebibliography}

\end{document}